\documentclass[11p,a4paper]{article}
\usepackage{amsfonts}
\usepackage{amssymb}
\usepackage{amsthm}
\usepackage{amsmath}
\usepackage{graphicx}

\usepackage{mathrsfs}
\usepackage{abstract}
\usepackage{color}
\usepackage{cite}
\usepackage{physics}
\usepackage[shortlabels]{enumitem}
\usepackage{empheq}
\usepackage{appendix}
\usepackage{mathrsfs}
\usepackage{bm} 
\usepackage[colorlinks=true,citecolor=red]{hyperref}
\hypersetup{linkcolor=blue}
\graphicspath{ {./result/} }
\usepackage{titling}
\usepackage{authblk}
\usepackage{geometry}
\usepackage{fancyhdr}
\usepackage{lineno}
\usepackage{titlesec}
\titleformat{\section}[block]{\centering\Large\bfseries}{\thesection}{1em}{}
\titleformat{\subsection}[block]{\centering\large\bfseries}{\thesubsection}{1em}{}
\titleformat{\subsubsection}[block]{\centering\normalsize\bfseries}{\thesubsubsection}{1em}{}

\usepackage{ulem}
\newcommand{\red}[1]{\textcolor{black}{#1}}

\newtheorem{theorem}{\textbf{Theorem}}[section]
\newtheorem{lemma}{\textbf{Lemma}}[section]
\newtheorem{proposition}{\textbf{Proposition}}[section]

\newtheorem{remark}{\textbf{Remark}}[section]

\providecommand{\keywords}[1]{\textbf{\textit{Keywords---}} #1}
\allowdisplaybreaks[4]

\def\Pi{\mathbf{\psi}}

\def\R{{\mathbb R}}

\def\N{{\mathbb N}}
\def\T{{\mathbb  T}}

\usepackage{tikz}
\usetikzlibrary{calc,arrows,intersections}
\title{\bf{Large-Time Behavior of Perturbations on the Rayleigh–Taylor Equilibria in the Boussinesq Equations}}

\author{
Song Jiang$^{a}$
\thanks{jiang@iapcm.ac.cn}
\quad\quad
Quan Wang$^{b}$
\thanks{Corresponding author:xihujunzi@scu.edu.cn }
\\ \footnotesize $^a$
 LCP, Institute of Applied Physics and Computational Mathematics,
\\\footnotesize Huayuan Road 6,
Beijing,100088, China
  \\ \footnotesize $^{b}$ College of Mathematics, Sichuan University,
  \footnotesize
 Chengdu, Sichuan, 610065,  China
}
\medskip

\begin{document}
\maketitle
\begin{abstract}
We study the two‑dimensional viscous Boussinesq equations, which model stratified flows in a circular domain under the influence of a general gravitational potential $f$.
First, we show that the Boussinesq equations admit steady‑state solutions only in the form of hydrostatic equilibria,  $(\mathbf{u},\rho,p)=(0,\rho_s,p_s)$, where the pressure gradient satisfies
$\nabla p_s=-\rho_s\nabla f$. Moreover, the relation between $\rho_s$ and $f$ is constrained by
 $\left(\partial_{y}\rho_{s},-\partial_{x}\rho_{s}\right)\cdot\left(\partial_{x}f,\partial_{y}f\right)=0$, 
which allows us to write $\nabla\rho_{s}=h\left(x,y\right)\nabla f$ for some scalar function $h\left(x,y\right)$.
Second, we prove that any hydrostatic equilibrium $(0,\rho_s,p_s)$
 is linearly unstable if $h(x_0,y_0)>0$ at some point $(x,y)=(x_0,y_0)$.
This instability coincides with the classical Rayleigh–Taylor instability. Third, by employing a series of regularity estimates, we reveal that although the presence
of the Rayleigh-Taylor instability makes perturbations around the unstable equilibrium grow
exponentially in time, the system ultimately converges to a state of hydrostatic equilibrium.
The analysis is carried out for perturbations about an arbitrary hydrostatic equilibrium, covering both stable and unstable configurations. Finally, we derive a necessary and sufficient condition on the initial density perturbation under which the density converges to a profile of the form $-\gamma f+\beta$ with constants $\gamma,\beta>0$. This result underscores the system’s inherent tendency to settle into a hydrostatic state, even in the presence of Rayleigh–Taylor instability.
\\
\end{abstract}

\newpage
\tableofcontents

\newpage
\section{Introduction}
Consider the two-dimensional viscous Boussinesq equations
   \begin{align}\label{main}
\begin{cases}
\frac{\partial \mathbf{u}}{\partial t}+ (\mathbf{u} \cdot \nabla )\mathbf{u} = \nu \Delta \mathbf{u}-
\frac{1}{\rho^*} \nabla P-\frac{\rho}{\rho^*}\nabla f,\quad \mathbf{x}\in \Omega,
 \\ \frac{\partial \rho}{\partial t}+(\mathbf{u}\cdot  \nabla)\rho
 =0, \quad \mathbf{x}\in \Omega,
\\ \nabla \cdot  \mathbf{u}=0, \quad \mathbf{x}\in \Omega,
\end{cases}
\end{align}
where $ \mathbf{u}=(u_1,u_2)$ is the velocity vector field,
 $\rho$ is the scalar density, $\rho^*$ is the constant density determined by
 a pressure $p^*$ and $f$ through the balance equation $\nabla p^*=-\rho^*\nabla f $,
$\nu$ is a positive constant which models viscous dissipation, and
 $f$ is a given potential function modeling gravity. It generally satisfies
the equation  $\Delta f=0$, and $\Omega$ is a doamin on
$\R^2$. 
The Boussinesq system can be used for modeling annular fluid systems,  such as planetary rings, accretion disks, the global atmosphere near the equator
and the fluid
between
two cylinders and homogeneous in the z-direction, which is derived from the inhomogeneous Navier-Stokes system
   \begin{align}\label{main-2}
\begin{cases}
\left(\frac{\rho}{\rho^*}+1\right)\frac{\partial \mathbf{u}}{\partial t}+ \left(\frac{\rho}{\rho^*}+1\right)(\mathbf{u} \cdot \nabla )\mathbf{u}
= \nu \Delta \mathbf{u}-  \frac{1}{\rho^*}\nabla P-\left(\frac{\rho}{\rho^*}+1\right)\nabla f,\quad \mathbf{x}\in \Omega,
 \\ \frac{\partial \rho}{\partial t}+(\mathbf{u}\cdot  \nabla)\rho  =0,\quad \mathbf{x}\in \Omega,
\\ \nabla \cdot  \mathbf{u}=0,\quad \mathbf{x}\in \Omega,
\end{cases}
\end{align}
by employing the Boussinesq approximation. That is, to derive \eqref{main}, one only needs to omit those terms containing $\rho/\rho^*$
except the one in the gravity term $-\rho\nabla f/\rho^*$. The Boussinesq system \eqref{main} can also be viewed as a coupled system of the continuity and momentum equations, with the coupling achieved solely through the gravity term 
$-\rho\nabla f/\rho^*$.
Please note that if the density is replaced by the temperature $\theta$, the gravity term $-\rho\nabla f/\rho^*$ in the Boussinesq system should be modified to the term $\theta\nabla f/\rho^*$, representing the buoyancy force. 

The system of Boussinesq equations is potentially relevant to the study of atmospheric and oceanographic dynamics,
as well as other astrophysical dynamics where rotation and stratification are predominant factors \cite{Majda-2003,Pedlosky1987, Smyth-2019}.
For instance, in the field of geophysics and fluid dynamics, the Boussinesq system is employed to model large-scale atmospheric and oceanic flows
that are responsible for phenomena such as cold fronts and jet streams  \cite{Pedlosky1987},  cloud rolls \cite{Wang2019}, large scale
circulation \cite{Ma20102,wang20192} and thermolhaline circulation
\cite{Ma2010}. In classical fluid mechanics, the system \eqref{main} with diffusion
is utilized for understanding the convection pattern in  buoyancy-driven flows \cite{Dijkstra2013,Majda-2002,Sengul2013}. It characterizes the motion of an incompressible, inhomogeneous, viscous fluid under the influence of gravitational forces, which can be instrumental in studying convective instability \cite{Smyth-2019}.

Beyond its intrinsic physical significance, the Boussinesq system is recognized for preserving crucial characteristics of the Euler and Navier-Stokes equations, such as the vortex stretching mechanism, where the increase in vorticity is contingent upon the temporal accumulation of $\nabla \mathbf{u}$, akin to the effect in three-dimensional incompressible flows \cite{Majda-2002}. Adding to its rich dynamics, when the viscosity $\nu=0$, the model aligns with the three-dimensional Euler equations for axisymmetric swirling flows for $r>0$ \cite{Katsaouni-2024,Majda-2002}, underscoring its capacity to bridge analyses of fluid motion across dimensions. Given these profound connections to fundamental fluid models, the questions of global well-posedness and potential finite-time blow up of its classical solutions have attracted considerable research interest. Regarding these issues, Chae and Nam \cite{Chae-Nam1997} proved local existence and uniqueness of smooth solutions for $\nu=0$ on $\mathbb{R}^2$ and derived a blow-up criterion, while Hou and Li \cite{Hou-li2005} demonstrated global well-posedness for the Cauchy problem with constant gravity $\nabla f=(0,g)^T$ and showed that solutions with initial data in $H^m(\mathbb{R}^2)$, $m\geq 3$, do not develop finite-time singularities. The question of global regularity versus singularity with $\nu=0$ remains a significant open problem in mathematical fluid mechanics, with further studies available in \cite{Adhikari-2014,Chae-2014,Weinan-1994,Hmidi-2010,Lai-2011,Lorca-1999}.
 
Before detailing the main objectives of this article, we review key mathematical results on the long‑time behavior near stable steady‑state solutions of the Boussinesq system \eqref{main}. The steady‑state solutions with zero velocity, called hydrostatic equilibria, are of the form $(\mathbf{0},\rho_s,p_s)$, where $\rho_s$ and $p_s$ satisfy the hydrostatic balance $\nabla p_s = -\rho_s \nabla f$.
In the case of constant vertical gravity $\nabla f = (0,g)$, the (stable) linear hydrostatic equilibrium
\[
(\mathbf{0},\rho_s,p_s) = \bigl(\mathbf{0},\, -ay + b,\; \tfrac{ag}{2}y^2 - gby + c\bigr),
\]
with constants $a>0$, $b$, $c$, describes a motionless fluid in which density decreases linearly with height (a lighter fluid overlying a denser one). Intuitively, this configuration is expected to be both linearly and nonlinearly stable. A substantial body of work has studied the dynamics of the constant‑gravity Boussinesq system \eqref{main} around this linear equilibrium. On a general bounded domain $\Omega$, 
Doering et al. \cite{Charles-R2018} proved that the velocity perturbation decays to zero in $H^1(\Omega)$, which implies that the perturbation equations approach the hydrostatic balance $\nabla p_s = -\rho_s \nabla (gy)$ in the $H^{-1}$ norm. Subsequent studies have focused on more specific geometries. Castro et al. \cite{Castro2019} established asymptotic stability for a special class of perturbations on the two‑dimensional strip $\Omega = \mathbb{T} \times (0,1)$, where the usual viscous dissipation $\nu \Delta \mathbf{u}$ is replaced by $-\mathbf{u}$. On the two‑dimensional periodic domain $\Omega = \mathbb{T}^2$, Tao et al. \cite{Lizheng-2020} obtained an explicit algebraic convergence rate of perturbations to the linear hydrostatic equilibrium, under the assumption that the perturbations satisfy suitable mean‑zero conditions.
Employing the constant‑gravity Boussinesq system with only vertical dissipation, Lai et al.\cite{Lai2021} established global stability of the linear hydrostatic equilibrium and obtained an explicit algebraic decay rate for perturbations on the whole plane $\mathbb{R}^2$.
Adhikari et al.~\cite{Adhikari2022} studied a model with horizontal dissipation and vertical thermal diffusion, establishing global stability and explicit algebraic decay rates on the unbounded strip $\Omega = \mathbb{T} \times \mathbb{R}$.  Extending this work, Jang and Kim~\cite{Jang-2023} derived the first sharp decay estimates for perturbations of the linear hydrostatic equilibrium on the two‑dimensional periodic domain $\mathbb{T}^2$, under general dissipation and thermal diffusion.
For additional results concerning the global stability of the linear hydrostatic equilibrium in the Boussinesq system \eqref{main}, we refer to \cite{Dong-2022, Ji-2022, Kang-2024, BCN} and the references therein.

For unbounded domains, the system \eqref{main} also permits steady‑state solutions with non‑zero velocity. The stability of stably stratified steady states with non‑vanishing velocity—particularly Couette flow—has attracted considerable mathematical interest due to its wide applicability in shear‑driven phenomena. Here we mention some of the research findings in this area for readers' reference:
1) Linear inviscid damping for Couette flow: Yang and Lin
\cite{Yang-lin2018} demonstrated linear inviscid damping for Couette flow within the 2D Boussinesq system with $\nu=0$,$\nabla f=(0,g)^T$ and on $\Omega=\T\times \R$; 2) Stability of Couette flow in the Boussinesq system:
Masmoudi et al. \cite{Nader-zhao-2022} established the stability of Couette flow for the 2D Boussinesq system with $\nabla f=(0,g)^T$ and on $\Omega=\T\times \R$ and for the initial perturbation in Gevrey-$1/s$ with $1/3<s\leq 1$; 3) Existence and stability of traveling solutions:
Zillinger \cite{Zillinger-2023} proved the stability and existence of time-dependent traveling wave solutions in any neighborhood of Couette flow
 and hydrostatic balance (with respect to local norms) for the 2D Boussinesq system with $\nabla f=(0,g)^T$.

In contrast to the extensively studied stably stratified case (density decreasing with height), rigorous mathematical results on the Rayleigh–Taylor instability—where a denser fluid overlies a lighter one (density increasing with height)—are relatively scarce. In 2001, Cherfils–Clérouin et al. \cite{CCLR-2001} considered the full equations \eqref{main-2} with $\nabla f = (0,g)^\top$ and $\nu\equiv 0$ on $\Omega = \mathbb{T} \times \mathbb{R}$. They rigorously established the linear instability of a general hydrostatic equilibrium $(\mathbf{0},\rho_s,p_s)$ under the condition $\frac{d\rho_s}{dy}\big|_{y=y_0} > 0$.
Subsequent advances include: Guo and Huang \cite{huang2003} extended this linear instability to nonlinear instability in the Hadamard sense and for the same setting; For the inviscid 2D Boussinesq system on $\Omega = \mathbb{T} \times \mathbb{R}$, Bedrossian et al. \cite{Bedrossian-2023} rigorously proved a shear‑buoyancy instability under the classical Miles–Howard condition on the Richardson number; In the viscous case ($\nu > 0$) on a bounded domain, Jiang and Jiang \cite{Jiang2014-ad} established nonlinear instability
in the Hadamard sense for profiles satisfying $\frac{d\rho_s}{dy}\big|_{y=y_0} > 0$
and nonlinear stability for profiles satisfying $\frac{d\rho_s}{dy}\big|_{y=y_0} \equiv h_0<0$,
using the full system \eqref{main-2} with $\nabla f = (0,g)^\top$.

To date, To date, nonlinear instability results 
in the Hadamard sense describe only short‑time growth; 
no rigorous mathematical theory exists for the long‑time dynamics near an unstable hydrostatic equilibrium in which denser fluid overlies lighter fluid. 
Nevertheless, if the two‑dimensional Boussinesq system \eqref{main}---or the full equations \eqref{main-2}---admits only hydrostatic equilibria as steady states (thereby precluding convergence to stable shear flows or other non‑trivial stable steady states), 
then even for a profile $\rho_s$ satisfying $\frac{d\rho_s}{dy}\big|_{y=y_0} > 0$, which is nonlinearly unstable in the Hadamard sense, 
the system should eventually relax to some hydrostatic equilibrium. 
This intuitive picture is supported by numerical simulations of Lee et al.~\cite{Lee2011}.
Motivated by the nonlinear instability analyses \cite{huang2003, Jiang2014-ad} and the numerical evidence \cite{Lee2011}, 
the present work aims to advance the theory of nonlinear instability and long‑time dynamics for system \eqref{main} in several key directions:
\begin{enumerate}
   \item [\rm{1)}] Under physically appropriate boundary conditions, we prove that the two‑dimensional Boussinesq system \eqref{main} admits only hydrostatic equilibria as steady‑state solutions. Specifically, every steady state is of the form $(\mathbf{u}, \rho, p) = (\mathbf{0}, \rho_s, p_s)$, where the density profile $\rho_s$ and the gravitational potential $f$ are related by $
\nabla \rho_s = h(x,y) \, \nabla f$ for some function $h(x,y)$.
Thus, convergence to shear flows or other non‑zero‑velocity steady states is ruled out.
\item [\rm{2)}] 
Our second objective is to establish both linear and nonlinear instability of hydrostatic equilibria $(\mathbf{0},\rho_s,p_s)$ under the condition $h(x_0,y_0)> 0$. This would constitute a substantial generalization of the earlier nonlinear instability results \cite{huang2003, Jiang2014-ad}, which were restricted to the case of a uniform vertical gravitational potential $f = gy$. In that special setting, the relation $\nabla \rho_s = h(x,y) \nabla f$ reduces to
\[
\frac{d\rho_s}{dy} = gh(x,y) \equiv gh(y),
\]
and the instability condition $h(x_0,y_0)> 0$ becomes $\frac{d\rho_s}{dy}\big|_{y=y_0} > 0$.
\item[\rm{3)}] We investigate the long‑time behavior of perturbations about arbitrary hydrostatic equilibria.
Specifically, we prove that the velocity perturbation converges to zero in the $H^1(\Omega)$‑norm as $t \to \infty$, irrespective of whether the equilibrium is stably stratified (a lighter fluid above a denser one) 
or Rayleigh–Taylor unstable (a denser fluid above a lighter one).
\item[\rm{4)}] We characterize the convergence of the density profile for solutions to \eqref{main-2} with $\mathbf{g} = \nabla f \neq \mathbf{0}$. Specifically, we derive necessary and sufficient conditions under which the density $\rho(t)$  approaches a linear steady‑state profile of the form $\rho_s = -\gamma f + \beta$, where $\gamma, \beta > 0$.
\end{enumerate}

Eigenvalue analysis provides an effective tool for studying the long‑time behavior of perturbations around stably stratified linear equilibria (where density decreases with height). 
This method, however, fails for analyzing the long‑time dynamics around Rayleigh–Taylor equilibria (a denser fluid overlying a lighter one). To investigate the large‑time dynamics of perturbations about such unstable equilibria in the two‑dimensional Boussinesq system \eqref{main} with a general non‑zero force field $\nabla f$, 
one may instead employ energy methods to establish decay of the velocity perturbation $\mathbf{u}$ in $H^1(\Omega)$. For unstable profiles, however, there exist no positive constants $d_1, d_2$ such that
\[
-\frac{d_1}{\rho^*}\int_{\Omega} \rho \nabla f \cdot \mathbf{u} \, dx \, dy
- d_1 \int_{\Omega} h(x,y) \, \rho \mathbf{u} \cdot \nabla f \, dx \, dy = 0.
\]
Consequently, the cancellation-of-cross-terms technique employed in previous works (e.g., \cite{Charles-R2018,Jiang2014-ad}) is no longer applicable for proving uniform $H^1$ bounds and decay. This fundamental obstruction thus necessitates the development of new energy‑estimate techniques tailored to the unstable regime.

In order to prove that the two‑dimensional Boussinesq system \eqref{main} on a bounded domain $\Omega$—under appropriate boundary conditions—admits only hydrostatic equilibria as steady‑state solutions, we work on a circular domain. For definiteness, we impose a free‑slip condition on the outer boundary and a Navier‑slip condition on the inner boundary (if present). This setup simplifies the presentation while preserving the essential analytical structure; all arguments in this article extend naturally to other smooth bounded domains with analogous boundary conditions.


\subsection{The Steady States}
In the real world, fluid flows typically occur within confined domains and are subject to boundary constraints, giving rise to boundary value problems. In the field of geophysical and astrophysical fluid dynamics, when examining fluid motions modeled by the Boussinesq system within a circular domain, it is more physically accurate to replace the gravity term $-\rho /\rho^*(0,g)^T$
by $-\rho\nabla f/\rho^*$. Here, $f$ is a potential function that describes the gravitational force. This modification better captures the effects of spatially varying gravitational forces or buoyancy effects in such geometries.

 We consider the steady-state solutions of the Boussinesq system \eqref{main} in a circular region on
$\R^2$ shown in \autoref{region},  which are determined by the following system:
    \begin{align}\label{main-s}
\begin{cases}
 (\mathbf{u} \cdot \nabla )\mathbf{u} = \nu \Delta \mathbf{u}-
\frac{1}{\rho^*} \nabla P-\frac{\rho}{\rho^*}\nabla f,\quad \mathbf{x}\in \Omega,\\
(\mathbf{u}\cdot  \nabla)\rho
 =0, \quad \mathbf{x}\in \Omega,
\\ \nabla \cdot  \mathbf{u}=0, \quad \mathbf{x}\in \Omega,
\end{cases}
\end{align}
 which is subject to the following boundary conditions
  \begin{align}\label{cond-1}
  &\mathbf{u} \cdot   \mathbf{n}|_{x^2+y^2=a^2}=0,\quad
  \mathbf{u} \cdot   \mathbf{n}|_{x^2+y^2=b^2}=\nabla \times   \mathbf{u}|_{x^2+y^2=b^2}=0,\\ \label{cond-2}
   &\left[\left(-\frac{p}{\rho^*}\mathbf{I} +\nu\left (\nabla \mathbf{u}+\left(\nabla \mathbf{u}\right)^{Tr} \right)\right)\cdot\mathbf{n}\right]
   \cdot   \mathbf{\tau}|_{x^2+y^2=a^2}
  +\alpha  \mathbf{u}  \cdot   \mathbf{\tau}|_{x^2+y^2=a^2}=0,
  \end{align}
  where $\mathbf{n}$ is the outward normal vector to the boundary $\partial\Omega$, 
  $\mathbf{\tau}$ is the tangent vector, and $\alpha $ is a given function.
 Please note that \eqref{cond-1} represents
 the stress-free boundary condition, while \eqref{cond-2} is the
 Navier-slip boundary condition which is a chosen interpolation between stress-free and no-slip boundary conditions.  It is more realistic than stress-free boundary conditions when boundaries are rough \cite{2007Handbook}. The conditions \eqref{cond-1}-\eqref{cond-2} are physically valid when the boundary $x^2+y^2=a^2$ is rough and 
 the outer boundary $x^2+y^2=b^2$ is free from any applied loads or constraints.
 
 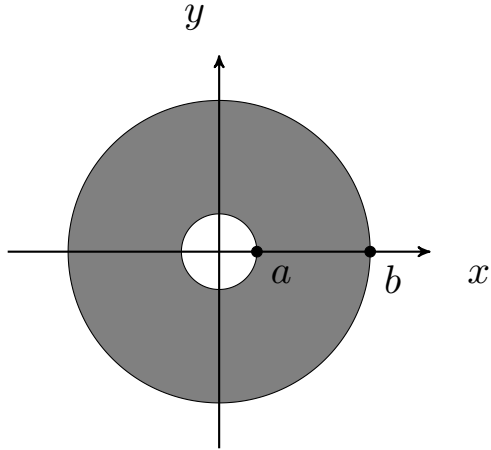
\begin{figure}[tbh]
\centering 
        \begin{tikzpicture}[>=stealth',xscale=1,yscale=1,every node/.style={scale=1.5}]
\fill[domain = -2:360,gray][samples = 200] plot({2*cos(\x)}, {2*sin(\x)});
        \fill[domain = -2:360,white][samples = 200] plot({1/2*cos(\x)}, {1/2*sin(\x)});
\draw [thick,->] (0,-2.6) -- (0,2.6) ;
\draw [thick,->] (-2.8,0) -- (2.8,0) ;
\node [below right] at (3.1,0) {$x$};
\node [below right] at (1/2,0) {$a$};
\node [below right] at (2,0) {$b$};
\node [left] at (0,3.1) {$y$};
\draw[domain = -2:360][samples = 200] plot({2*cos(\x)}, {2*sin(\x)});
\draw[domain = -2:360][samples = 200] plot({1/2*cos(\x)}, {1/2*sin(\x)});
\draw[fill,black] (0.5,0) circle [radius=2pt];
\draw[fill,black] (2,0) circle [radius=2pt];
\end{tikzpicture}
\caption{Schematic representation of the domain $\Omega$}\label{region} 
\end{figure}

Under the boundary conditions \eqref{cond-1}-\eqref{cond-2}, we indeed can establish that
hydrostatic equilibria are the only stationary solutions of the equations.
\begin{lemma}\label{steady-lemma}
Let
$(\mathbf{u}, \rho)$ be any classical solution of the equations \eqref{main-s} subject to the boundary conditions \eqref{cond-1}-\eqref{cond-2},
if $\nu a^{-1}+\alpha\geq0$, then the following results are established.
\begin{align}
\mathbf{u}=0,\quad  \rho=\rho_s(x,y),
\end{align}
and $\rho_s(x,y)$ solves
\begin{align}\label{rhos}
\quad \partial_y\rho_s\partial_xf-
\partial_x\rho_s\partial_yf=0.
\end{align}
\end{lemma}
 \begin{proof}
Using $ \mathbf{u}\cdot  \mathbf{n}|_{\partial\Omega}=0$, the second and third equations of \eqref{main-s}, we get
 \begin{align*}
  \begin{aligned}
 0&=
 \int_{\Omega}\left( (\mathbf{u}\cdot  \nabla)\rho\right)
 f\,dx\,dy
 =
 - \int_{\Omega}\rho
 \nabla f\cdot \mathbf{u}
 \,dx\,dy.
  \end{aligned}
 \end{align*}
   The boundary conditions and the
divergence-free condition allow us to remove the terms
  \[
    2 \int_{\Omega}\left( (\mathbf{u} \cdot \nabla )\mathbf{u}
 \right)\cdot\mathbf{u}\,dx\,dy=-
  \int_{\partial\Omega}\abs{\mathbf{u}}^2
  \mathbf{u}\cdot\mathbf{n}\,ds
 =0.
  \]
By leveraging \autoref{lemma-grad-1-1}, we multiply the first two equations in \eqref{main-s} by the vector field $\mathbf{u}$ in the $L^2$
  inner product space, yielding
  \begin{align}\label{s-proof-2}
    \begin{aligned}
  & \rho^*\int_{\Omega}\left( (\mathbf{u} \cdot \nabla )\mathbf{u}
 \right)\cdot\mathbf{u}\,dx\,dy-\mu
    \int_{\Omega}\Delta \mathbf{u}\cdot\mathbf{u}\,dx\,dy=-
   \int_{\Omega}\rho
\nabla f\cdot \mathbf{u}
 \,dx\,dy\Rightarrow \\
&\mu
\int_{\Omega}\abs{\nabla\mathbf{u}}^2\,dx\,dy
   \red{ + \frac{\mu}{b}\int_{\partial B(0,b)}\left(u_{\mathbf{\tau}}\right)^2 \,ds
}  +\int_{\partial B(0,a)}\left(\mu a^{-1}+ \rho^*\alpha\right)\left(u_{\mathbf{\tau}}\right)^2 \,ds
\\&=-
   \int_{\Omega}\rho
\nabla f\cdot \mathbf{u}
 \,dx\,dy=0.
   \end{aligned}
  \end{align}
  Then, due to the condition $ \mathbf{u}\cdot\mathbf{n} |_{x^2+y^2=b^2}=0$, one can infer from \eqref{s-proof-2} that
   $\mathbf{u}=0$.
 \end{proof}
\begin{remark}
 For any $\gamma,\beta\in \R$,
$\left(\mathbf{u},\rho,p\right)=\left(\mathbf{0},-\gamma f(x,y)+\beta,\frac{1}{2}\left(\gamma f(x,y)-\beta\right)^2\right)$ is a family of steady-state solutions of
the evolution system \eqref{main} .
\end{remark}
\begin{remark}
The relation \eqref{rhos} says that the vector fields $\nabla\rho_s$ and $\nabla f$
are pointwise parallel. Hence, there exists a scale function $h(x,y)$ such that 
$\nabla\rho_s=h(x,y)\nabla f$.
\end{remark}
\begin{remark}
If $f=g\ln\sqrt{x^2+y^2}+c$, we have \[
 -\rho\nabla f=-\rho g\left(\frac{x}{x^2+y^2},\frac{y}{x^2+y^2}\right)^T,
 \]
 which is the gravity in the negative of $r$-direction. For any function
$\rho_s(x,y)=\tilde{\rho}(\sqrt{x^2+y^2})$ where $\tilde{\rho}(s)$
is a function defined on $[a,b]$, one can see that $(\rho_s,f)$ solves \eqref{rhos}.\end{remark}

     \subsection{Instability and open problem}
     
Having established that hydrostatic equilibria are the only admissible steady states under the given boundary conditions \eqref{cond-1}-\eqref{cond-2}, a natural subsequent question arises concerning their instability. In particular, one must distinguish between configurations that remain stable under small perturbations and those that exhibit instability, potentially leading to complex transient dynamics. The following subsection addresses this issue by presenting rigorous linear and nonlinear instability criteria for hydrostatic profiles that satisfy a density-gradient condition.

  \begin{theorem}\label{theorem-3}
Suppose $h(x,y),f(x,y) \in C^{2}(\overline{\Omega})$ and $\Delta f=0$. If the profile $\rho_s \in C^{1}(\overline{\Omega}) $ satisfies
$\nabla \rho_s =h(x,y) \nabla f$ with $h(x_0,y_0) >0$ at $(x_0,y_0)\in\Omega$, then the steady-state $(\mathbf{0}, \rho_s)$
of  the equations \eqref{main} subject to the  boundary conditions \eqref{cond-1}-\eqref{cond-2} is linearly unstable. If there exists a constant $h_0$ such that
$h(x,y) < h_0<0$, then the steady-state $(\mathbf{0}, \rho_s)$
is linearly stable.
\end{theorem}
    \begin{remark}
    The proof of  \autoref{theorem-3} is presented in Section 3.1.
     \end{remark}
     
    \begin{remark}
 Based on \autoref{theorem-3},
 for any fixed number $\gamma ,\beta \in \R$,
the steady-state $(\mathbf{0},\rho_s)=(\mathbf{0}, -\gamma f(x,y)+\beta)$ is linearly stable if $\gamma>0$,
 while it is linearly unstable if $\gamma<0$. \autoref{theorem-4} gives more precise description on the stability of
 $(\mathbf{0}, -\gamma f(x,y)+\beta)$ with $\gamma >0$.
  \end{remark}
 \begin{remark}
 The linear instability
  described in \autoref{theorem-3}, where $\rho_s$ satisfying
$\nabla \rho_s =h(x,y) \nabla f$ with $h(x_0,y_0) >0$ at $(x_0,y_0)\in\Omega$,
is referred to as the Rayleigh-Taylor instability. This extends the classical case of 
$ \nabla f=(0,g)^T$.
The presence of the Rayleigh-Taylor instability implies that the solutions of the equations \eqref{main} near such unstable hydrostatic equilibrium will grow over time.
   \end{remark}

We now present our results on the nonlinear instability of general linearly unstable density profiles $\rho_s$. To this end, we consider perturbations about a smooth hydrostatic equilibrium $(\mathbf{u}_s, \rho_s) = (\mathbf{0}, \rho_s)$. Introducing the perturbative variables $ (\mathbf{u} - \mathbf{u}_s, \rho - \rho_s, p - p_s)$, we obtain the following system, which governs the evolution of the perturbation:
   \begin{align}\label{213}
\begin{cases}
\frac{\partial \mathbf{u}}{\partial t}+ (\mathbf{u} \cdot \nabla )\mathbf{u} = \nu \Delta \mathbf{u}-
\frac{1}{\rho^*} \nabla P-\frac{\rho}{\rho^*}\nabla f
 \\ \frac{\partial \rho}{\partial t}+(\mathbf{u}\cdot  \nabla)\rho
 =-(\mathbf{u}\cdot  \nabla  \rho_s),
\\ \nabla \cdot  \mathbf{u}=0,
\end{cases}
\end{align}
subject to the boundary conditions \eqref{cond-1}-\eqref{cond-2} and
 the initial data
\begin{align}\label{ii-c}
\mathbf{u}|_{t=0}=\mathbf{u}_0,\quad \rho|_{t=0}=\rho_0.
\end{align}

\begin{theorem}\label{hadamardyiyixia0202}[{\bf Nonlinear instability} ]
If there exists a point \(\left(x_{0},y_{0}\right)\in \Omega\) such that \(h\left(x_{0},y_{0}\right)>0\), the steady-state solution \(\left(\mathbf{0},\rho_{s}\right)\) is unstable in Hadamard sense. That is, there exist two constants  \(\epsilon\) and \(\delta_{0}\), and functions \(\left(\mathbf{u}_{0},\varrho_{0}\right)\)\(\in \left[H^{2}\left(\Omega\right)\right]^2\times \left[H^{1}\left(\Omega\right)\cap L^{\infty}\left(\Omega\right)\right]\), such that for any \(\delta^{*}\in\left(0,\varrho_{0}\right)\) and initial data \(\left(\mathbf{u}_{0}^{\delta^{*}},\varrho_{0}^{\delta^{*}}\right):=\delta^{*}\left(\mathbf{u}_{0},\varrho_{0}\right)\), the strong solution \(\left(\mathbf{u}^{\delta^{*}},\varrho^{\delta^{*}}\right)\in C\left(0,T_{\text{max}},\left[H^{1}\left(\Omega\right)\right]^2\times L^{2}\left(\Omega\right)\right)\) of the problem \eqref{213}-\eqref{ii-c} subject to
the conditions \eqref{cond-1}-\eqref{cond-2} and
initial data \(\left(\mathbf{u}_{0}^{\delta^{*}},\varrho_{0}^{\delta^{*}}\right)\) satisfies 
\[
\left\|\varrho^{\delta^{*}}\left(T^{\delta^{*}}\right)\right\|_{L^{p}\left(\Omega\right)}\geq \epsilon,~
\left\|\mathbf{u}^{\delta^{*}}\left(T^{\delta^{*}}\right)\right\|_{L^p\left(\Omega\right)}\geq \epsilon,\quad
p\in[1,\infty],
\]
for some escape time \(0<T^{\delta^{*}}<T_{max}\), where \(T_{max}\) is the maximal existence time of \(\left(\mathbf{u}^{\delta^{*}},
\varrho^{\delta^{*}}\right)\).
\end{theorem}

    \begin{remark}
    One can adapt the method of \rm{\cite{Jiang2014-ad}} to prove \autoref{hadamardyiyixia0202}.
We omit the details.
     \end{remark}
     
The instability results presented above highlight the possibility of transient growth in perturbations around certain hydrostatic equilibria in a short time. Nevertheless, the question of how such perturbations evolve over large time scales—remains open. To address this, we now turn to the regularity and asymptotic behavior of solutions to  the nonlinear system \eqref{213}
subject to the boundary conditions \eqref{cond-1}-\eqref{cond-2} and the initial data \eqref{ii-c}. The following subsection establishes global-in-time regularity estimates and describes the large-time dynamics of the system  \eqref{main}
 on a circular domain or other bounded domains, demonstrating that despite possible short-term instability, the velocity field eventually decays and the density profile approaches a steady state.

    \subsection{Regularity and large time behavior of solutions}
 \begin{theorem}\label{theorem-1}
 Let $q\geq 4$. For the problem \eqref{213}-\eqref{ii-c}  with initial data $(\mathbf{u}_0,\rho_0) \in H^{2}\left(\Omega\right)\times L^{q}\left(\Omega\right)$
 satisfying the incompressibility $\nabla\cdot \mathbf{u}_0=0$ and 
 subject to the boundary conditions
\eqref{cond-1}-\eqref{cond-2}, if $\alpha \in  W^{1,\infty}\left(\partial B(0,a)\right)$,
 $\nabla f \in  W^{1,\infty}\left(\Omega\right)$ and $\left(\nu a^{-1}+\alpha\right)\geq 0$,
then we get that the solutions of \eqref{213} satisfy the following properties:
\begin{subequations}
     \begin{align}\label{theorem-1-conc-1}
&\mathbf{u} \in L^{\infty}\left(\left(0,\infty\right);H^2(\Omega)\right)
\cap L^{p}\left(\left(0,\infty\right);W^{1,p}(\Omega)\right),\quad 2\leq p<\infty,\\
\label{theorem-1-conc-2}
&\mathbf{u}_t \in
L^{\infty}\left(\left(0,\infty\right);L^2(\Omega)\right)
\cap L^{2}\left(\left(0,\infty\right);H^{1}(\Omega)\right),\\
\label{theorem-1-conc-3}
&p\in L^{\infty}\left(\left(0,\infty\right);H^1(\Omega)\right),\\
\label{theorem-1-conc-4}
&\rho \in L^{\infty}\left(\left(0,\infty\right);L^s(\Omega)\right),\quad 1\leq s\leq q.
 \end{align}
   \end{subequations}
If $\alpha \in  W^{2,\infty}\left(\partial B(0,a)\right)$
 and $(\mathbf{u}_0,\rho_0) \in H^{3}\left(\Omega\right)\times W^{1,q}\left(\Omega\right)
$ with $q\geq 2$,  then for any $T>0$,
\begin{subequations}
     \begin{align}\label{theorem-1-conc-11}
&\mathbf{u} \in L^{2}\left(\left(0,T\right);H^3(\Omega)\right)
\cap L^{\frac{2p}{p-2}}\left(\left(0,T\right);W^{2,p}(\Omega)\right),\quad 2\leq p<\infty,\\
\label{theorem-1-conc-12}
&\rho \in L^{\infty}\left(\left(0,T\right);W^{1,s}(\Omega)\right),\quad 1\leq s\leq q.
 \end{align}
   \end{subequations}
 \end{theorem}
    \begin{remark}
The global existence of weak and strong
solutions of the problem \eqref{213}-\eqref{ii-c} subject to the boundary conditions \eqref{cond-1}-\eqref{cond-2} and
on the circular domain can be proved by employing the standard method, see \rm{\cite{Lai-2011}}
and the references therein.
    \end{remark}
    
  \begin{theorem}\label{theorem-2}[\textbf{Large time behavior}]
 Under the conditions of \autoref{theorem-1},
  for any $1\leq r<\infty, (\gamma, \beta)\in\R^+\times \R$,
the solutions of the problem \eqref{213}-\eqref{ii-c} subject to the boundary conditions \eqref{cond-1}-\eqref{cond-2}
satisfy the following  coclusioins:
  \begin{enumerate}
  \item [\rm{(1)}] The solutions of the problem \eqref{213}-\eqref{ii-c}
satisfy the following asymptotic properties:
\begin{subequations}
     \begin{align}\label{theorem-2-conc-1}
&\norm{\mathbf{u} }_{W^{1,r}}\to 0,\quad t\to \infty,\\
\label{theorem-2-conc-2}
&\norm{\mathbf{u}_t }_{L^{2}}\to 0,\quad t\to \infty, \\
\label{theorem-2-conc-3}
&\Delta \mathbf{u}\rightharpoonup 0 \quad \text{in} \quad L^2(\Omega),\quad t\to \infty,\\
\label{theorem-2-conc-4}
&\nabla P+\rho \nabla f\rightharpoonup 0 \quad \text{in} \quad L^2(\Omega),\quad t\to \infty.
 \end{align}
 \end{subequations}
 \item [\rm{(2)}] For any $\gamma>0$ and $\beta \in \R$, we have
 \begin{subequations}
     \begin{align}
\label{theorem-22-conc-4-1}
&\int_{\Omega}\varrho f \,d\mathbf{x}\to I_1,\quad t\to \infty,\\
\label{theorem-22-conc-4}
&\norm{\rho +\rho_s+\gamma f-\beta}_{L^{2}}^2 \to I_2,\quad t\to \infty,
 \end{align}
 \end{subequations}
   where $I_1$ and $I_2$ are two constants which satisfy 
   \begin{subequations}
     \begin{align}\label{I-11a}
              \begin{aligned}
     & I_1\leq \frac{\norm{\mathbf{u}_0 }_{L^{2}}^2}{2}+
     \int_{\Omega}\varrho_0 f \,d\mathbf{x},\\ 
   &0\leq I_2\leq
   \gamma \norm{\mathbf{u}_0 }_{L^{2}}^2+\norm{\varrho_0+\rho_s+\gamma f-\beta}_{L^{2}}^2,\\
&\norm{\varrho_0+\rho_s+\gamma f(x,y)-\beta}_{L^2}^2-I_2
  =2\gamma \left(\int_{\Omega}\varrho _0f \,d\mathbf{x}- I_1\right).
 \end{aligned}
  \end{align}
 \end{subequations}
\end{enumerate}
 \end{theorem}

  \begin{theorem}\label{theorem-3-3}
\rm{[\textbf{Large time behavior}]}
 Under the conditions of \autoref{theorem-1},
  for the solution $( \mathbf{u},\varrho)$ of the problem\eqref{213}-\eqref{ii-c}, the following two conclusions hold:  
  \begin{enumerate}
   \item [\rm{(1)}] The following asymptotical result holds
    \begin{subequations}
         \begin{align}\label{theorem-2-conc-2-2-1}
    &\int_\Omega \varrho f d\mathbf{x}\to 0,\quad t\to \infty,
     \end{align}
      \end{subequations}
    if and only if there exist $\gamma>0$ and $\beta$ such that
     \begin{subequations}
  \begin{align}\label{dineg-1}
\begin{aligned}
\norm{\varrho_0+\rho_s+\gamma f(x,y)-\beta}_{L^2}^2
-\lim_{t\to +\infty}\norm{\varrho +\rho_s+\gamma f-\beta}_{L^{2}}^2
 =2\gamma \int_\Omega \varrho _0f d\mathbf{x}.
\end{aligned}
  \end{align}
   \end{subequations}
\item [\rm{(2)}] 
The following asymptotical result holds
    \begin{subequations}
         \begin{align}\label{theorem-2-conc-2-2}
    &\norm{\varrho +\rho_s-(-\gamma f+\beta)}_{L^{2}} \to 0,\quad t\to \infty,
     \end{align}
      \end{subequations}
    if and only if there exist $\gamma>0$ and $\beta$ such that
     \begin{subequations}
  \begin{align}\label{dineg}
\begin{aligned}
 \lim_{t\to\infty}2\gamma 
\int_\Omega (\varrho_0 -\varrho(t) )f d\mathbf{x}
= \norm{\varrho_0+\rho_s+\gamma f(x,y)-\beta}_{L^2}^2.
\end{aligned}
  \end{align}
   \end{subequations}
\end{enumerate}
 \end{theorem}

      \begin{remark}
 Based on \autoref{theorem-3}, if we consider the perturbation for any unstable steady-state $(\mathbf{0},\rho_s(x,y))$, although
the presence of the Rayleigh-Taylor instability means that there exists a linear or nonlinear growth with time of the perturbation \rm{\cite{Hwang-2003,Jiang2014-ad,JJ2013,Mao-2024}},
our result (\autoref{theorem-2}) indeed shows that the
fluid eventually comes into the vicinity of a stable state of hydrostatic equilibrium where the velocity becomes zero and the density
increases along the direction of gravity. Particularly,
the density can approach $ -\gamma f+\beta$ itself
as long as the condition \eqref{dineg} holds.
  \end{remark}
  
For stable steady profile $\rho_s$, we show that the convergence \eqref{theorem-2-conc-4} can be improved to
the strong case when $f=gr$. The linearized equations are given by
   \begin{align}\label{213-2}
\begin{cases}
\frac{\partial \mathbf{u}}{\partial t}= \nu \Delta \mathbf{u}-
\frac{1}{\rho^*} \nabla P-\frac{\rho}{\rho^*} \nabla gr,
 \\ \frac{\partial \rho}{\partial t}= -(\mathbf{u}\cdot  \nabla  \rho_s),
\\ \nabla \cdot  \mathbf{u}=0,
\end{cases}
\end{align}
 which is subject to the following boundary conditions
  \begin{align}\label{cond-1-1}
  &\mathbf{u} \cdot   \mathbf{n}|_{x^2+y^2=a^2}=0,\quad
  \mathbf{u} \cdot   \mathbf{n}|_{x^2+y^2=b^2}=\nabla \times   \mathbf{u}|_{x^2+y^2=a^2,b^2}=0,
  \end{align}

  \begin{theorem}\label{theorem-4}
For the equations \eqref{213-2} with initial data $(\mathbf{u}_0,\rho_0) \in H^{2}\left(\Omega\right)
 \times H^{1}\left(\Omega\right)$ satisfying incompressibility
$\nabla\cdot \mathbf{u}_0=0$
and subject to the boundary conditions \eqref{cond-1}, let $\rho_s$ be the profile, such that
 $\nabla \rho_s(x,y)=h(x,y)g\nabla r $ and $h (x,y)\leq h_0<0$.
The solutions of the equations \eqref{213-2} satisfy the following properties:
\begin{subequations}
\begin{align}\label{theorem-4-conc-1}
&\mathbf{u} \in L^{\infty}\left(\left(0,\infty\right),H^2(\Omega)\right)
\cap L^{p}\left(\left(0,\infty\right),H^{1}(\Omega)\right),\quad 2\leq p<\infty,\\
\label{theorem-4-conc-2}
&\mathbf{u}_t \in
L^{\infty}\left(\left(0,\infty\right);H^1(\Omega)\right)
\cap L^{2}\left(\left(0,\infty\right);H^{2}(\Omega)\right),\\
\label{theorem-4-conc-3}
&P\in L^{\infty}\left(\left(0,\infty\right);H^1(\Omega)\right),\\
\label{theorem-4-conc-4}
&\rho \in L^{\infty}\left(\left(0,\infty\right);L^2(\Omega)\right).
 \end{align}
   \end{subequations}
Furthermore, if $h (x,y)\equiv h_0<0$, we have 
\begin{subequations}
     \begin{align}\label{theorem-3-conc-1}
&\norm{\mathbf{u} }_{H^{2}}\to 0,\quad t\to \infty,\\
\label{theorem-3-conc-2}
&\norm{\mathbf{u}_t }_{H^{1}}\to 0,\quad t\to \infty,\\
\label{theorem-3-conc-3}
&\norm{\nabla P+\rho\nabla (gr)}_{L^{2}}\to 0 ,\quad t\to \infty.
 \end{align}
 \end{subequations}
 \end{theorem}
    \begin{remark}
    Our method also works for the problem \eqref{main} on bounded $C^{3}$-domain with the stress-free boundary conditions or periodic boundary conditions or purely Navier-slip boundary condition. Hence,
     \autoref{theorem-1}-\autoref{theorem-3-3} also hold for any bounded $C^{3}$-domain with the stress-free boundary conditions or periodic boundary conditions or purely Navier-slip boundary condition.
  \end{remark}

\subsection{Main technical difficulties behind the proof}
The key estimate for the proof of \autoref{theorem-1}--\ref{theorem-2} lies in controlling the norm $
\norm{\mathbf{u}}_{L^{\infty}\left((0,\infty);H^{1}(\Omega)\right)}$.
This control serves as a fundamental prerequisite not only for establishing the asymptotic behavior of $|\nabla \mathbf{u}|_{L^2}^2$, but also for that of $\norm{\mathbf{u}_t}_{L^2}^2$, $\norm{\Delta\mathbf{u}}_{L^2}^2$ and the term $\nabla P+\rho \nabla f$.

For the non-slip boundary condition on a bounded domain, 
Kukavica et al.~\cite{Kukavica2023} and Aydin et al.~\cite{Aydin2025,Aydin2025-1}
developed a method to establish the boundedness of $\norm{\mathbf{u}}_{L^{\infty}\left((0,\infty);H^{1}(\Omega)\right)}$. Their method, however, does not apply to the boundary conditions \eqref{cond-1}--\eqref{cond-2}. For the Navier-slip boundary condition on a bounded domain,
Bleitner et al.~\cite{Bleitner2025} employed a maximum principle to obtain
$\norm{\nabla^{\perp}\cdot\mathbf{u}}_{L^{\infty}\left((0,\infty);L^{p}(\Omega)\right)}$, which in turn 
yields the boundedness of
$\norm{\mathbf{u}}_{L^{\infty}\left((0,\infty);H^{1}(\Omega)\right)}$. 
However, to deduce the asymptotic convergence of $\norm{\mathbf{u}}_{H^{1}}$, one must also control the pressure term. The obstacle lies in the fact that the trace theorem does not directly provide the estimate
$\mu \int_{\partial B(0,a)}P \tau \cdot\nabla \left(\left(2a^{-1}+\frac{\alpha}{\nu}\right)u_{\tau}\right)\,ds\leq
  \norm{P}_{H^{1}}  \norm{\mathbf{u}}_{H^{1}}$
which is needed for the analysis. 
For the mixed boundary conditions \eqref{cond-1}--\eqref{cond-2}, we develop a new approach
to establish both the boundedness $\norm{\mathbf{u}}_{L^{\infty}\left((0,\infty);H^{1}(\Omega)\right)}$
and the asymptotic convergence of $\norm{\mathbf{u}}_{H^{1}}$. More precisely, we firstly obtain an energy identity, by which we further establish that
\[
\mathbf{u}\in L^{\infty}\left((0,\infty);L^2(\Omega)\right)
\cap L^{2}\left((0,\infty);H^1(\Omega)\right).
\]
Second, through a careful analysis, we derive the following key inequality:
\[
\begin{aligned}
&\int_{\Omega}\left(\rho+\rho_s\right)\mathbf{u}\cdot \nabla f\,dx\,dy+
\frac{\nu\rho^*}{2}\int_\Omega |\nabla \mathbf{u}|^2 \,dx\,dy
+\frac{\rho^*\nu}{2b}
\int_{\partial B(0,b)}\left(u_{\mathbf{\tau}}\right)^2 \,ds
\\&\quad+\frac{\rho^*}{2}
\int_{\partial B(0,a)}
\left(\nu a^{-1}+\alpha\right)\left(u_{\mathbf{\tau}}\right)^2 \,ds
 + \rho^*\int_0^t \|\mathbf{u}_s\|_{L^2}^2 \,ds 
 \\&\leq 
 \int_{\Omega}\left(\varrho_0+\rho_s\right)\mathbf{u}_0\cdot \nabla f\,dx\,dy +
\frac{\nu\rho^*}{2}\int_\Omega |\nabla \mathbf{u}_0|^2 \,dx\,dy\\
&\quad+\frac{\rho^*\nu}{2b}
\int_{\partial B(0,b)}\left(u_{\mathbf{\tau}}(0)\right)^2 \,ds
+\frac{\rho^*}{2}
\int_{\partial B(0,a)}
\left(\nu a^{-1}+\alpha\right)\left(u_{\mathbf{\tau}}(0)\right)^2 \,ds
 \\&\quad+\int_0^t
\left(
\norm{\nabla \mathbf{u}}_{L^2}^2
+\norm{\mathbf{u}}_{L^2\left(\partial B(0,b)\right)}^2
\right)
 h\left(
\norm{\nabla \mathbf{u}}_{L^2}^2
+\norm{\mathbf{u}}_{L^2\left(\partial B(0,b)\right)}^2
\right)\,ds,
 \end{aligned}
\]
where $h(z)$ is a positive function given by $h(z)=C_0\sqrt{z}+C_0z^{3/4}+C_0z+C_0$, and $C_0$ is a positive constant depending on the initial data. This key inequality enables us to establish both
$\norm{\mathbf{u}}_{L^{\infty}\left((0,\infty);H^{1}(\Omega)\right)}$
and the asymptotic convergence of $\norm{\mathbf{u}}_{H^{1}}$. We believe that
our method for controlling $\norm{\mathbf{u}}_{L^{\infty}\left((0,\infty);H^{1}(\Omega)\right)}$
is not only applicable to similar problems with a non-slip boundary condition,
but also effective for analogous problems with a purely Navier-slip boundary condition.

The key challenge in the proof of \autoref{theorem-4} is to establish the uniform boundedness in time of $\norm{\nabla \rho}_{L^2}$.
A direct estimate of $\norm{\nabla \rho}_{L^2}$ does not yield an inequality sufficient to prove its uniform boundedness.
Instead, we estimate the weighted norms $\norm{(x^2+y^2)\nabla \rho}_{L^2}$ and $\norm{(x^2+y^2)\nabla^{\perp}\cdot\mathbf{u}}_{L^2}$.
This approach allows us to derive the following inequality:
\[
\begin{aligned}
&\frac{1}{2}\frac{d}{dt}
\int_a^b\left(
\int_0^{2\pi}
\left(
\left(\partial_r\rho \right)^2+
\frac{1}{r^2} \left(\partial_{\theta}\rho\right)^2
+\omega^2
\right)
r^3\,d\theta
\right)\,dr
+\frac{d}{dt}
\int_a^b\int_0^{2\pi}
\frac{r\rho^2}{-2h_0}
\,d\theta\,dr
\\
&+\frac{\nu}{2}
\int_a^b\left(
\int_0^{2\pi}
\left(
\left(\partial_r\omega\right)^2+
\frac{1}{r^2}
\left(\partial_{\theta}\omega \right)^2
\right)
r^3\,d\theta
\right)\,dr
\leq C
\int_a^b\left(
\int_0^{2\pi}
\omega^2
r\,d\theta
\right)\,dr.
\end{aligned}
\]
From this inequality, the uniform boundedness in time of $\norm{\nabla \rho}_{L^2}$ follows.\subsection{Organization of the paper}

The rest of this article is arranged as follows: Section 2 gives regularity estimates for the Boussinesq equations. Proofs of the main theorems are given in Section 3.

 \section{Regularity Estimates}

  All $L^q$-norms of $(\rho+\rho_s)$ are conserved if $\mathbf{u}$ is sufficiently smooth, i.e. for any
  $1 \leq q < \infty$,
  \begin{align}\label{rhopp}
  \norm{\rho+\rho_s}_q= \norm{\rho_0+\rho_s}_q,
  \end{align}
  for all $t\geq 0$ if $\rho_0\in L^{q}(\Omega)$.
  In fact for the case of $q \in 2\N$, one can infer from 
the second equation of the system \eqref{213} that
  \begin{align*}
    \begin{aligned}
 \frac{d \norm{\rho+\rho_s}^q_{L^q}}{dt}&=
 \frac{d}{dt}
 \int_{\Omega}\abs{\rho+\rho_s}^q\,dx\,dy=-q
 \int_{\Omega}\abs{\rho+\rho_s}^{q-1}\mathbf{u}\cdot\nabla (\rho+\rho_s)\,dx\,dy\\&
 =- \int_{\Omega} (\rho+\rho_s)^{q}
 \mathbf{u}\cdot \mathbf{n}
 \,dx\,dy+
 \int_{\Omega} (\rho+\rho_s)^{q}
\nabla \cdot  \mathbf{u}
 \,dx\,dy=0,
   \end{aligned}
  \end{align*}
where we have used $\nabla \cdot  \mathbf{u}=0$ and the boundary condition \eqref{cond-1}.
Bearing in mind that the  $L^q$ -norm of $\rho+\rho_s$ is conserved, we obtain
\begin{align}
\norm{\rho}_{L^q}=\norm{\rho+\rho_s-\rho_s}_{L^q}
\leq \norm{\rho+\rho_s}_{L^q}+
\norm{\rho_s}_{L^q}=
\norm{\rho_0+\rho_s}_{L^q}+
\norm{\rho_s}_{L^q}<\infty .
\end{align}
This enables us to estimate the energy of the fluid, as demonstrated in the subsequent lemma.

\subsection{Stokes's Estimate }
\begin{lemma}\label{Stokes0725}[\textbf{Stokes' estimates}]. 
Suppose  $\left(\nu a^{-1}+\alpha\right)\geq 0$.
Consider the equation 
\begin{align}\label{biaozhunde0725}
\begin{cases}
\mu \Delta \mathbf{v} - \nabla P = \mathbf{f}, \\
\nabla \cdot \mathbf{v} = 0,
\end{cases}
\end{align}
subject to the boundary conditions \eqref{cond-1}-\eqref{cond-2}, where $\mathbf{f}=\left(f_{1},f_{2}\right)\in \left[L^{2}\left(\Omega\right)\right]^2$. Assume that 
$\left(\mathbf{v},P\right)\in \left[H^1\left(\Omega\right)\right]^2\times L^{2}\left(\Omega\right)$ is the weak solution to the above equation, then we have the following estimates 
\begin{align}\label{manzu0725}
    \left\|\nabla^2\mathbf{v}\right\|_{L^{2}\left(\Omega\right)}
    +\left\|\nabla P\right\|_{L^{2}\left(\Omega\right)}\leq C \left\|\mathbf{f}\right\|_{L^{2}\left(\Omega\right)}.
    \end{align}
\end{lemma}
\begin{proof}
We will construct a solution to the following equation 
\begin{align}\label{wufadingyi1020}
\begin{cases}
-\mu\Delta \mathbf{w}+\nabla P=\mathbf{f},
\\
\nabla\cdot \mathbf{w}=0,
\\
\left(\partial_{1}w_{2}-\partial_{2}w_{1}\right)|_{r=b}=\mathbf{w}\cdot\mathbf{n}|_{r=a,b}=0,\\
\left(\partial_{1}w_{2}-\partial_{2}w_{1}\right)|_{r=a}
=\left(\frac{2}{r}+\frac{\alpha}{\nu}\right)v_{\mathbf{\tau}}|_{r=a},
\\
\left(\partial_{1}w_{2}-\partial_{2}w_{1}\right)|_{r=b}
=\left(\partial_{1}w_{2}-\partial_{2}w_{1}\right)|_{r=b},
\end{cases}
\end{align}
where $\mathbf{w}=\left(w_{1},w_{2}\right)$ and $v_{\mathbf{\tau}}=w_{2}\cos{\theta}-w_{1}\sin{\theta}$. First, we consider the following auxiliary problem 
\begin{align}\label{fuzhu11020}
\begin{cases}
-\mu\Delta G=\partial_{1}f_{2}-\partial_{2}f_{1},
\\
G|_{\partial\Omega}=
\frac{b-r}{b-a}\left(\frac{2}{r}+\frac{\alpha}{\nu}\right)v_{\theta}
:=U\in L^{2}\left(\partial\Omega\right).
\end{cases}
\end{align}
From the elliptic theory, one can conclude that there exists a weak solution 
\(G\) satisfying 
\[
\left\|\nabla G\right\|_{L^{2}\left(\Omega\right)}
\leq C\left(\left\|\mathbf{f}\right\|_{L^{2}\left(\Omega\right)}+\left\|\mathbf{v}\right\|_{H^{1}\left(\Omega\right)}\right).
\]
Furthermore, we consider another auxiliary problem 
\begin{align}\label{fuzhu21020}
\begin{cases}
\Delta\Phi=G,
\\
\Phi|_{\partial\Omega}=0.
\end{cases}
\end{align}
Similarly, from the elliptic theory, there exists a \(\Phi\) satisfying the above equation and 
\[
\left\|\nabla^3\Phi\right\|_{L^{2}\left(\Omega\right)}\leq C\left\|\nabla G\right\|_{L^{2}\left(\Omega\right)}
\leq C\left(\left\|\mathbf{f}\right\|_{L^2\left(\Omega\right)}+\left\|\mathbf{v}\right\|_{H^{1}\left(\Omega\right)}\right).
\]

Let us denote $w_{1}=-\partial_{2}\Phi,~w_{2}=\partial_{1}\Phi$,
then 
\(\nabla\cdot\mathbf{w}=0\), \(\mathbf{w}\cdot\mathbf{n}|_{r=a,b}=0\) and $\left(\partial_{1}w_{2}-\partial_{2}w_{1}\right)|_{r=b}=0$. Furthermore, 
\(\Delta\Phi=\partial_{1}w_{2}-\partial_{2}w_{1}=G\). Thus, 
\(-\mu\Delta G=-\mu\Delta\left(\partial_{1}w_{2}-\partial_{2}w_{1}\right)=\partial_{1}f_{2}-\partial_{2}f_{1}\). Then, from the field theory(an irrotational field must be  potential), there exists a \(\nabla P\) such that
\[
-\mu\Delta \mathbf{w}+\nabla P=\mathbf{f}.
\]
From the above construction, one can easily verify that \(\left(\mathbf{v},P\right)\) satisfies the equation \eqref{wufadingyi1020}. And 
\begin{align}\label{one-22}
\left\|\nabla^2\mathbf{w}\right\|_{L^{2}\left(\Omega\right)}
+\left\|\nabla\mathbf{P}\right\|_{L^{2}\left(\Omega\right)}
\leq
C\left(\left\|\mathbf{f}\right\|_{L^2\left(\Omega\right)}+\left\|\mathbf{w}\right\|_{H^{1}\left(\Omega\right)}\right).
\end{align}

It is not hard to see that
\[
\begin{aligned}
&\nu \int_{\Omega}\abs{\nabla \mathbf{u}}^2\,dx\,dy+ \red{  \frac{\nu}{b}\int_{\partial B(0,b)}\left(u_{\mathbf{\tau}}\right)^2 \,ds
}
+\int_{\partial B(0,a)}
\left(\nu a^{-1}+\alpha\right)\left(u_{\mathbf{\tau}}\right)^2 \,ds=\int_{\Omega}\mathbf{f}\cdot \mathbf{u}
 \,dx\,dy.
\end{aligned}
\]

An application of \autoref{lemma-grad}, there exist constants $C$ and $D$ such that
\[
\begin{aligned}
D\int_{\Omega}\abs{\mathbf{w}}^2\,dx\,dy\leq
&\int_{\Omega}\abs{\nabla \mathbf{w}}^2\,dx\,dy+ \red{  \frac{\nu}{2b}\int_{\partial B(0,b)}\left(w_{\mathbf{\tau}}\right)^2 \,ds
}
\\&+\frac{1}{2}\int_{\partial B(0,a)}
\left(\nu a^{-1}+\alpha\right)\left(w_{\mathbf{\tau}}\right)^2 \,ds\leq C\int_{\Omega}\abs{\mathbf{f}}^2
 \,dx\,dy,
\end{aligned}
\]
which yields that $\left\|\mathbf{w}\right\|_{H^{1}\left(\Omega\right)}
\leq C\left\|\mathbf{f}\right\|_{L^2\left(\Omega\right)}$.
This, together with \eqref{one-22}, gives \eqref{manzu0725}.

If we can show that \(\mathbf{v}=\mathbf{w}\) almost everywhere, then the proof is finished. Thus, we will show the uniqueness of solution to the following equation 
\begin{align}\label{weiyixing1020}
\begin{cases}
-\mu\Delta \mathbf{H}+\nabla \Theta =\mathbf{0},
\\
\nabla\cdot \mathbf{H}=0,
\\
\mathbf{H}\cdot\mathbf{n}|_{r=a,b}=0,~\left(\partial_{1}H_{2}-\partial_{2}H_{1}\right)|_{r=a,b}=0.
\end{cases}
\end{align}
To this end, we consider the following 
auxiliary problem 
\begin{align}\label{renyi1020}
\begin{cases}
-\mu\Delta \mathbf{Q}+\nabla \widetilde{\Theta}=\widetilde{\mathbf{F}},
\\
\nabla\cdot \mathbf{Q}=0,~
\mathbf{Q}\cdot\mathbf{n}|_{r=a,b}=0,~\left(\partial_{1}Q_{2}-\partial_{2}Q_{1}\right)|_{r=a,b}=0,
\end{cases}
\end{align}
where \(\widetilde{\mathbf{F}}\in \left[L^{2}\left(\Omega\right)\right]^2\) is arbitrary. Following the same steps for the auxiliary problems 
\eqref{fuzhu11020} and \eqref{fuzhu21020} where 
\(U\equiv 0\) and \(\mathbf{f}\) is replaced by \(\widetilde{\mathbf{F}}\), one can obtain the strong solution \(\left(\mathbf{Q},\widetilde{\Theta}\right)\in \left[H^{2}\left(\Omega\right)\right]^2\times \left[L^{2}\left(\Omega\right)\right]^2\) to the problem \eqref{renyi1020}. Multiplying \(\eqref{weiyixing1020}_{1}\) by \(\mathbf{Q}\) and 
\(\eqref{renyi1020}_{1}\) by \(\mathbf{H}\), then integrating the results over \(\Omega\) by parts, respectively, one can find that $
0=\int_{\Omega}\widetilde{\mathbf{F}}\cdot \mathbf{H}dx$,
which implies \(\mathbf{H}=\mathbf{0}\) almost everywhere since 
\(\widetilde{\mathbf{F}}\) is arbitrary. Thus, the uniqueness for the problem \eqref{weiyixing1020} holds. Therefore, we can conclude that \(\mathbf{v}=\mathbf{w}\). Thus, \(\left(\mathbf{v},P\right)\) satisfies the estimate \eqref{manzu0725}. 
\end{proof}

\subsection{Estimate of $\norm{\mathbf{u}}_{L^{\infty}\left((0,\infty);H^1(\Omega)\right)}$}

\begin{lemma}\label{ll2}
For the solution $( \mathbf{u},\rho)$ of the problem \eqref{213}-\eqref{ii-c} 
subject to the boundary conditions \eqref{cond-1}-\eqref{cond-2} and
with initial data $(\mathbf{u}_0,\rho_0)\in L^2\left(\Omega\right)\times L^2\left(\Omega\right) $,
if $\alpha \in  L^{\infty}\left(\partial B(0,a)\right)$,
 $\nabla f \in  L^{\infty}\left(\Omega\right)$
and
  $\left(\nu a^{-1}+\alpha\right)\geq 0$, we have
\begin{align}\label{result-1}
\mathbf{u}\in L^{\infty}\left((0,\infty);L^2(\Omega)\right)
\cap L^{2}\left((0,\infty);H^1(\Omega)\right).
\end{align}
\end{lemma}
\begin{proof}
By multiplying the first equation of \eqref{213} by $\mathbf{u}$ and applying \autoref{lemma-grad-1-1}, we obtain
\begin{align}\label{one-proof-1}
\begin{aligned}
&\frac{d}{dt}\int_{\Omega}\rho^*\frac{\abs{\mathbf{u}}^2}{2}\,dx\,dy
+\rho^*\nu\int_{\Omega}\abs{\nabla \mathbf{u}}^2\,dx\,dy
+\frac{\rho^*\nu}{ b}\int_{\partial B(0,b)}\left(u_{\mathbf{\tau}}\right)^2 \,ds
\\&
+\int_{\partial B(0,a)}
\left(\rho^*\nu a^{-1}+\rho^*\alpha\right)\left(u_{\mathbf{\tau}}\right)^2 \,ds
=-
 \int_{\Omega}\rho
\nabla f\cdot \mathbf{u}
 \,dx\,dy.
\end{aligned}
\end{align}

We also note that 
\begin{align*}
\begin{aligned}
\frac{d}{dt} 
 \int_{\Omega}\rho f
 \,dx\,dy
&=
 \int_{\Omega}\rho
\nabla f\cdot \mathbf{u}
 \,dx\,dy+ \int_{\Omega}\rho_s
\nabla f\cdot \mathbf{u}
 \,dx\,dy\\
 &=
  \int_{\Omega}\rho
\nabla f\cdot \mathbf{u}
 \,dx\,dy- \int_{\Omega}\nabla P_s\cdot \mathbf{u}
 \,dx\,dy= \int_{\Omega}\rho
\nabla f\cdot \mathbf{u}
 \,dx\,dy.
 \end{aligned}
\end{align*}
This, together with \eqref{one-proof-1}, yields that
\begin{align}\label{one-proof-1-n}
\begin{aligned}
&\int_{\Omega}\rho^*\abs{\mathbf{u}}^2\,dx\,dy
+ 2\int_{\Omega}\rho f
 \,dx\,dy
+2\rho^*\nu\int_0^t\left(\int_{\Omega}\abs{\nabla \mathbf{u}(\tau)}^2\,dx\,dy\right)\,d\tau
\\&+\frac{2\rho^*\nu}{ b}
\int_0^t\left(
\int_{\partial B(0,b)}\left(u_{\mathbf{\tau}}\right)^2 \,ds
\right)\,d\tau
+2\rho^*\int_0^t\left(\int_{\partial B(0,a)}
\left(\nu a^{-1}+\alpha\right)\left(u_{\mathbf{\tau}}\right)^2 \,ds
\right)\,d\tau\\&=
\int_{\Omega}\rho^*\abs{\mathbf{u}_0}^2\,dx\,dy
+ 2\int_{\Omega}\rho_0 f
 \,dx\,dy.
\end{aligned}
\end{align}
With the aid of \autoref{lemma-grad}, the result
\eqref{result-1} follows from the identity \eqref{one-proof-1-n}.
\end{proof}

\subsection{Estimate of $\norm{\mathbf{u}}_{L^{\infty}\left((0,\infty);H^{1}(\Omega)\right)}$}

\begin{lemma}\label{lemma-grad-1-2}
 For the solution $( \mathbf{u},\rho)$ of the problem \eqref{213}-\eqref{ii-c} 
subject to the boundary conditions \eqref{cond-1}-\eqref{cond-2} and
with initial data $\mathbf{u}_0 \in H^{1}\left(\Omega\right) $, if $\alpha \in  L^{\infty}\left(\partial B(0,a)\right)$,
 $\nabla f \in  L^{\infty}\left(\Omega\right)$ and $\left(\nu a^{-1}+\alpha\right)\geq 0$, 
 we have
\begin{align}\label{infty}
\mathbf{u}\in L^{\infty}\left((0,\infty);H^{1}(\Omega)\right);\quad
\mathbf{u}_t\in L^{2}\left((0,\infty);L^{2}(\Omega)\right).
\end{align}
\end{lemma}

\begin{proof}
Using \autoref{lemma-grad-1-1},
testing $\eqref{213}_1$ by $\rho^*\mathbf{u}_t$ and integrating over $\Omega$, we have
\begin{align}\label{nablau-l2}
\begin{aligned}
&\nu\rho^*\frac{1}{2} \frac{d}{dt} \int_\Omega |\nabla \mathbf{u}|^2 dx 
+\frac{\rho^*}{2} \frac{d}{dt}
\int_{\partial B(0,a)}
\left(\nu a^{-1}+\alpha\right)\left(u_{\mathbf{\tau}}\right)^2 \,ds
\\&+\frac{\rho^*\nu}{2b} \frac{d}{dt}
\int_{\partial B(0,b)}\left(u_{\mathbf{\tau}}\right)^2 \,ds
+ \rho^*  \|\mathbf{u}_t\|_{L^2}^2 \\
&\leq C\norm{\mathbf{u}}_{L^4}^2
\norm{\nabla \mathbf{u}}_{L^4}^2-\int_{\Omega}\rho \mathbf{u}_{t}\cdot \nabla f\,dx\\
&\leq C\norm{\mathbf{u}}_{L^4}^2
\norm{\nabla \mathbf{u}}_{L^2}\norm{\nabla^2 \mathbf{u}}_{L^2}
+C\norm{\mathbf{u}}_{L^4}^2
\norm{\nabla \mathbf{u}}_{L^2}^2
-\int_{\Omega}\rho \mathbf{u}_{t}\cdot \nabla f\,dx:=I_1+I_2+I_3.
\end{aligned}
\end{align}
To estimate the term \( I_1 \), it is necessary to control the \( L^2 \)-norm of $\nabla^2 \mathbf{u}$. Using \autoref{Stokes0725}, we have
\[
\begin{aligned}
\norm{\nabla^2 \mathbf{u}}_{L^2}
&\leq C\norm{\mathbf{u}_t}_{L^2}
+C\norm{\mathbf{u}}_{L^4}
\norm{\nabla \mathbf{u}}_{L^4}
+C\norm{\rho  \nabla f}_{L^2}\\
&\leq C\norm{\mathbf{u}_t}_{L^2}
+C\norm{\mathbf{u}}_{L^4}
\norm{\mathbf{\nabla u}}_{L^2}^{\frac{1}{2}}
\left(\norm{\nabla^2 \mathbf{u}}_{L^2}
^{\frac{1}{2}}
+\norm{\nabla\mathbf{u}}_{L^2}^{\frac{1}{2}}
\right)
+C\norm{\rho  \nabla f}_{L^2}\\
&\leq C\norm{\mathbf{u}_t}_{L^2}
+C\norm{\mathbf{u}}_{L^4}^2
\norm{\nabla \mathbf{u}}_{L^2}+
1/2\norm{\nabla^2 \mathbf{u}}_{L^2}
\\&\quad+C\norm{\mathbf{u}}_{L^4}
\norm{\nabla \mathbf{u}}_{L^2}
+\norm{\rho  \nabla f}_{L^2}.
\end{aligned}
\]
This implies that 
\begin{align}\label{delta-u}
\begin{aligned}
\norm{\nabla^2 \mathbf{u}}_{L^2}
&\leq C\norm{\mathbf{u}_t}_{L^2}
+C\norm{\mathbf{u}}_{L^4}^2
\norm{\nabla \mathbf{u}}_{L^2}
+C\norm{\mathbf{u}}_{L^4}
\norm{\nabla \mathbf{u}}_{L^2}
+\norm{\rho  \nabla f}_{L^2}.
\end{aligned}
\end{align}
With the aid of \autoref{lemma-grad}, we have
\begin{align}
\norm{\mathbf{u}}_{L^4}^2\leq 
C\norm{\mathbf{u}}_{L^2}
\norm{\mathbf{u}}_{H^1}\leq C
\norm{\mathbf{u}}_{L^2}\sqrt{
\norm{\nabla \mathbf{u}}_{L^2}^2
+\norm{\mathbf{u}}_{L^2\left(\partial B(0,b)\right)}^2}.
\end{align}
We deduce the following bound for \( I_2 \):
\begin{align}\label{ln-nablau-2}
\norm{\mathbf{u}}_{L^4}^2\norm{\nabla \mathbf{u}}_{L^2}^2\leq
C\norm{\mathbf{u}}_{L^2}\left(
\sqrt{
\norm{\nabla \mathbf{u}}_{L^2}^2
+\norm{\mathbf{u}}_{L^2\left(\partial B(0,b)\right)}^2}
\right)^3.
\end{align}
This implies that 
\begin{align}\label{ln-nablau-2-2}
\begin{aligned}
I_1\leq & \frac{\rho^*}{2}\norm{\mathbf{u}_t}_{L^2}^2
+C\norm{\mathbf{u}}_{L^4}^4
\norm{\nabla \mathbf{u}}_{L^2}^2+C\norm{\mathbf{u}}_{L^4}^3
\norm{\nabla \mathbf{u}}_{L^2}^2
+C_0\norm{\mathbf{u}}_{L^4}^2
\norm{\nabla \mathbf{u}}_{L^2}
\\\leq & \frac{\rho^*}{2}\norm{\mathbf{u}_t}_{L^2}^2
+C 
\norm{\mathbf{u}}_{L^2}^2
\left(
\norm{\nabla \mathbf{u}}_{L^2}^2
+\norm{\mathbf{u}}_{L^2\left(\partial B(0,b)\right)}^2
\right)^2
\\&
+C 
\norm{\mathbf{u}}_{L^2}^{3/2}
\left(
\norm{\nabla \mathbf{u}}_{L^2}^2
+\norm{\mathbf{u}}_{L^2\left(\partial B(0,b)\right)}^2
\right)^{7/4}
+C 
\norm{\mathbf{u}}_{L^2}
\left(
\norm{\nabla \mathbf{u}}_{L^2}^2
+\norm{\mathbf{u}}_{L^2\left(\partial B(0,b)\right)}^2
\right).
\end{aligned}
\end{align}

To handle \( I_3 \), we simplify its expression via \(\eqref{213}_2\), deviating from the direct use of Hölder's inequality employed in prior studies:
\begin{align}\label{I_3}
\begin{aligned}
-\int_{\Omega}\rho \mathbf{u}_{t}\cdot \nabla f\,dx\,dy
&=-\partial_t
\int_{\Omega}\left(\rho+\rho_s\right)\mathbf{u}\cdot \nabla f\,dx\,dy
+
\int_{\Omega}\rho_t\mathbf{u}\cdot \nabla f\,dx\,dy \\
&=-\partial_t
\int_{\Omega}\left(\rho+\rho_s\right)\mathbf{u}\cdot \nabla f\,dx\,dy
\\&\quad-\int_{\Omega}\left(\mathbf{u}\cdot \nabla \left(\rho+\rho_s\right)\right)
\left(
\ \mathbf{u}\cdot \nabla f\right)\,dx\,dy \\
&=-\partial_t
\int_{\Omega}\left(\rho+\rho_s\right)\mathbf{u}\cdot \nabla f\,dx\,dy
+\int_{\Omega}\left(\rho+\rho_s\right)\mathbf{u}\cdot
\nabla\left(
\ \mathbf{u}\cdot \nabla f\right)\,dx\,dy .
\end{aligned}
\end{align}

Substituting \eqref{delta-u}-\eqref{I_3} into \eqref{nablau-l2}, we have
\begin{align}\label{new-asymptotical}
\begin{aligned}
&\partial_t
\int_{\Omega}\left(\rho+\rho_s\right) \mathbf{u}\cdot \nabla f\,dx\,dy+
\nu\rho^*\frac{1}{2} \frac{d}{dt} \int_\Omega |\nabla \mathbf{u}|^2\,dx\,dy
\\&
+\frac{\rho^*\nu}{2b} \frac{d}{dt}
\int_{\partial B(0,b)}\left(u_{\mathbf{\tau}}\right)^2 \,ds
+\frac{\rho^*}{2} \frac{d}{dt}
\int_{\partial B(0,a)}
\left(\nu a^{-1}+\alpha\right)\left(u_{\mathbf{\tau}}\right)^2 \,ds
+ \frac{\rho^*}{2} \|\mathbf{u}_t\|_{L^2}^2 \\ \le &
C 
\norm{\mathbf{u}}_{L^2}^2
\left(
\norm{\nabla \mathbf{u}}_{L^2}^2
+\norm{\mathbf{u}}_{L^2\left(\partial B(0,b)\right)}^2
\right)^2
+C 
\norm{\mathbf{u}}_{L^2}^{3/2}
\left(
\norm{\nabla \mathbf{u}}_{L^2}^2
+\norm{\mathbf{u}}_{L^2\left(\partial B(0,b)\right)}^2
\right)^{7/4}
\\&+C 
\norm{\mathbf{u}}_{L^2}
\left(
\norm{\nabla \mathbf{u}}_{L^2}^2
+\norm{\mathbf{u}}_{L^2\left(\partial B(0,b)\right)}^2
\right)+C\norm{\mathbf{u}}_{L^2}\left(
\sqrt{
\norm{\nabla \mathbf{u}}_{L^2}^2
+\norm{\mathbf{u}}_{L^2\left(\partial B(0,b)\right)}^2}
\right)^3
\\&+C\left(
\norm{\nabla \mathbf{u}}_{L^2}^2
+\norm{\mathbf{u}}_{L^2\left(\partial B(0,b)\right)}^2
\right)\\ \le &
C_0
\left(
\norm{\nabla \mathbf{u}}_{L^2}^2
+\norm{\mathbf{u}}_{L^2\left(\partial B(0,b)\right)}^2
\right)^2+C_0
\left(
\norm{\nabla \mathbf{u}}_{L^2}^2
+\norm{\mathbf{u}}_{L^2\left(\partial B(0,b)\right)}^2
\right)\\
&+C_0
\left(
\norm{\nabla \mathbf{u}}_{L^2}^2
+\norm{\mathbf{u}}_{L^2\left(\partial B(0,b)\right)}^2
\right)^{7/4}+C_0\left(
\sqrt{
\norm{\nabla \mathbf{u}}_{L^2}^2
+\norm{\mathbf{u}}_{L^2\left(\partial B(0,b)\right)}^2}
\right)^3.
\end{aligned}
\end{align}
Integrating the preceding inequality in time over the interval $[0, t]$, we obtain
\begin{align}\label{nablau-l2-detail}
\begin{aligned}
&\int_{\Omega}\left(\rho+\rho_s\right)\mathbf{u}\cdot \nabla f\,dx\,dy+
\frac{\nu\rho^*}{2}\int_\Omega |\nabla \mathbf{u}|^2 \,dx\,dy
+\frac{\rho^*\nu}{2b}
\int_{\partial B(0,b)}\left(u_{\mathbf{\tau}}\right)^2 \,ds
\\&+\frac{\rho^*}{2}
\int_{\partial B(0,a)}
\left(\nu a^{-1}+\alpha\right)\left(u_{\mathbf{\tau}}\right)^2 \,ds
 + \rho^*\int_0^t \|\mathbf{u}_s\|_{L^2}^2 \,ds 
 \\&\leq 
 \int_{\Omega}\left(\varrho_0+\rho_s\right)\mathbf{u}_0\cdot \nabla f\,dx\,dy +
\frac{\nu\rho^*}{2}\int_\Omega |\nabla \mathbf{u}_0|^2 \,dx\,dy\\
&+\frac{\rho^*\nu}{2b}
\int_{\partial B(0,b)}\left(u_{\mathbf{\tau}}(0)\right)^2 \,ds
+\frac{\rho^*}{2}
\int_{\partial B(0,a)}
\left(\nu a^{-1}+\alpha\right)\left(u_{\mathbf{\tau}}(0)\right)^2 \,ds
 \\&+\int_0^t
\left(
\norm{\nabla \mathbf{u}}_{L^2}^2
+\norm{\mathbf{u}}_{L^2\left(\partial B(0,b)\right)}^2
\right)
 h\left(
\norm{\nabla \mathbf{u}}_{L^2}^2
+\norm{\mathbf{u}}_{L^2\left(\partial B(0,b)\right)}^2
\right)\,ds,
 \end{aligned}
\end{align}
where $h(z)$ is a positive function given by
\[
\begin{aligned}
h(z)=C_0\sqrt{z}+C_0z^{3/4}+C_0z+C_0.
 \end{aligned}
\]

Using the energy estimate \eqref{result-1}, inequality \eqref{nablau-l2-detail} implies
\begin{align}\label{nablau-l2-detail-3}
\int_\Omega |\nabla \mathbf{u}|^2  d\,dx\,dy
+
\int_{\partial B(0,b)}\left(u_{\mathbf{\tau}}\right)^2 \,ds
 \leq C_0 + \int_0^t \norm{\nabla \mathbf{u}}_{L^2}^2  h\left( \norm{\nabla \mathbf{u}}_{L^2}^2 \right) ds.
\end{align}
To facilitate the analysis, we introduce the following notation:
\[
y(t) = \int_\Omega |\nabla \mathbf{u}|^2  d\,dx\,dy
+
\int_{\partial B(0,b)}\left(u_{\mathbf{\tau}}\right)^2 \,ds=g(t) , \quad w(z) = \frac{2}{\nu\rho^*}  h(z).
\]
In terms of these functions, inequality \eqref{nablau-l2-detail-3} can be rewritten as
\[
y(t) \leq C_0 + \int_0^t g(s)  w(y(s))  ds:=\tilde{y}(t).
\]

Note that
\[
\tilde{y}'=g(t)  w(y(t)) =y(t)  w(y(t)) 
\leq y(t)  w(\tilde{y}(t)) ,
\]
from which we get that
\[
\frac{\tilde{y}'}{w(\tilde{y}(t)) }\leq y(t)\Rightarrow
\int_0^t\frac{\tilde{y}'(\tau)}{w(\tilde{y}(\tau)) }\,d\tau
=\int_{y(0)}^{y(t)}
\frac{1}{w(z) }\,dz
\leq \int_0^t y(\tau )\,d\tau.
\]

Define the function $G(z) = \int_{C_0}^z \frac{1}{w(s)}  ds$.
We then get that
\[
G(y(t))-G(y(0))\leq \int_0^t y(\tau )\,d\tau.
\]

A direct verification shows that $\int_{C_0}^{+\infty} \frac{1}{w(s)}  ds = +\infty$.
This means that
\[
y(t)\leq G^{-1}\left(
G(y(0))+\int_0^t y(\tau )\,d\tau
\right).
\]
Therefore, an application of
\eqref{one-proof-1-n} yields
\[
\begin{aligned}
& \int_\Omega |\nabla \mathbf{u}|^2  d\,dx\,dy
+\int_{\partial B(0,b)}\left(u_{\mathbf{\tau}}\right)^2 \,ds
\\&\leq G^{-1}\left(G(C_0)+\int_0^t y(s)\,ds\right)
\\&=G^{-1}\left(
\int_0^t\left(
\norm{\nabla \mathbf{u}}_{L^2}^2
+\norm{\mathbf{u}}_{L^2\left(\partial B(0,b)\right)}^2
\right)\,d\tau
\right)
\\&\leq G^{-1}\left(
\int_0^{+\infty}\left(
\norm{\nabla \mathbf{u}}_{L^2}^2
+\norm{\mathbf{u}}_{L^2\left(\partial B(0,b)\right)}^2
\right)\,d\tau
\right)<+\infty.
\end{aligned}
\]
This, combining with \eqref{nablau-l2-detail}, yields \eqref{infty}.

\end{proof}

\subsection{Estimate of $\norm{\mathbf{u}_t}_{ L^{2}\left((0,\infty);H^{1}(\Omega)\right)}$
and  $\norm{\mathbf{u}_t}_{L^{\infty}\left((0,\infty;,L^{2}(\Omega)\right)}$}

\begin{lemma}\label{utestimate}
For the problem \eqref{213}-\eqref{ii-c} subject to the boundary conditions \eqref{cond-1}-\eqref{cond-2} and with initial data $(\mathbf{u}_0,\rho_0) \in H^{2}\left(\Omega\right)
 \times L^{4}\left(\Omega\right)$, if $\alpha \in  W^{1,\infty}\left(\partial B(0,a)\right)$,
 $\nabla f \in  W^{1,\infty}\left(\Omega\right)$ and $\left(\nu a^{-1}+\alpha\right)\geq 0$,
we have
\begin{align}\label{tul2h1}
\mathbf{u}_t \in L^{\infty}\left((0,\infty);L^{2}(\Omega)\right)
\cap L^{2}\left((0,\infty);H^{1}(\Omega)\right).
\end{align}
\end{lemma}
\begin{proof} We replace $\rho$ and $P$ by $\rho-\rho_s$ and $P-P_s$ in the system \eqref{213}, where $\rho_s$ and $P_s$ solve
$\nabla P_s=-\rho_s\nabla f$. Then, by differentiating the corresponding equations and the conditions \eqref{cond-1}-
\eqref{cond-2} with respect to $t$, we observe that $\frac{\partial \mathbf{u}}{\partial t}$ satisfies the following system 
   \begin{align}\label{ut-1}
\begin{cases}
\mathbf{u}_{tt}+ (\mathbf{u}_t \cdot \nabla )\mathbf{u}
+ (\mathbf{u}\cdot \nabla )\mathbf{u}_t
= \nu \Delta \mathbf{u}_t-
\frac{1}{\rho^*} \nabla P_t-\frac{\rho_t}{\rho^*}\nabla f,
 \\ \rho_t+(\mathbf{u}\cdot  \nabla)\rho
 =0,
\\ \nabla \cdot  \mathbf{u}_t=0.
\end{cases}
\end{align}
The boundary conditions for the system \eqref{ut-1} are as follows:
  \begin{align}\label{cond-1-1}
  &\mathbf{u}_t \cdot   \mathbf{n}|_{x^2+y^2=a^2}=0,\quad
  \mathbf{u}_t \cdot   \mathbf{n}|_{x^2+y^2=b^2}=
  \nabla \times  \mathbf{u}_t
  =0,\\ \label{cond-2-1}
   &\left[\left(-\frac{P_t}{\rho^*}\mathbf{I} +\nu\left (\nabla \mathbf{u}_t+\left(\nabla \mathbf{u}_t\right)^{Tr} \right)\right)\cdot\mathbf{n}\right]
   \cdot   \mathbf{\tau}|_{x^2+y^2=a^2}
  +\alpha  \mathbf{u}_t  \cdot   \mathbf{\tau}|_{x^2+y^2=a^2}=0.
  \end{align}

Taking the $L^2$-inner product of \eqref{ut-1} with $\mathbf{u}_t $ and using
  \autoref{lemma-grad-1-1}, we infer that
  \begin{align}\label{ut-proof-1}
\begin{aligned}
&\frac{d}{dt}\int_{\Omega}\rho^*\frac{\abs{\mathbf{u}_t}^2}{2}\,dx\,dy
+\rho^*\nu\int_{\Omega}\abs{\nabla \mathbf{u}_t}^2\,dx\,dy
\\&
+\rho^*\int_{\partial B(0,a)}\left(\nu a^{-1}+\alpha\right)\left(\tau \cdot \mathbf{u}_t\right)^2 \,ds+
\frac{\rho^*\nu}{b}\int_{\partial B(0,b)}
\left(\tau \cdot \mathbf{u}_t\right)^2 \,ds
\\  &=- \int_{\Omega}\rho_t
\nabla f\cdot \mathbf{u}_t
 \,dx\,dy- \int_{\Omega}\rho
\mathbf{u}_t\cdot (\mathbf{u}_t\cdot\nabla )\mathbf{u}
 \,dx\,dy.
\end{aligned}
\end{align}
Substituting the second equation of \eqref{ut-1} into \eqref{ut-proof-1} and performing integration by parts gives
    \begin{align}\label{ut-proof-2}
\begin{aligned}
&\frac{d}{dt}\int_{\Omega}\rho^*\frac{\abs{\mathbf{u}_t}^2}{2}\,dx\,dy
+\nu\rho^*\int_{\Omega}\abs{\nabla \mathbf{u}_t}^2\,dx\,dy
\\&
+\red{\frac{\nu\rho^*}{b}\int_{\partial B(0,b)}\left(\tau \cdot \mathbf{u}_t\right)^2 \,ds}
+\rho^*\int_{\partial B(0,a)}\left(\nu a^{-1}+\alpha\right)\left(\tau \cdot \mathbf{u}_t\right)^2 \,ds
\\&=-
 \int_{\Omega}\rho ( \mathbf{u}\cdot\nabla)
(\nabla f\cdot \mathbf{u}_t)
 \,dx\,dy- \int_{\Omega}
\mathbf{u}_t\cdot (\mathbf{u}_t\cdot\nabla )\mathbf{u}
 \,dx\,dy.
\end{aligned}
\end{align}
By H\"older's, Ladyzhenskaya's and Young's inequalities, one deduces that
    \begin{align}\label{ut-proof-3}
&\begin{aligned}
 \int_{\Omega}\abs{
\mathbf{u}_t\cdot (\mathbf{u}_t\cdot\nabla )\mathbf{u}}
 \,dx\,dy\leq &\norm{
 \nabla \mathbf{u}}_{L^2}
 \norm{ \mathbf{u}_t}^2_{L^4}\leq
C\norm{
 \nabla \mathbf{u}}_{L^2}\norm{\mathbf{u}_t}_{L^2}\sqrt{
\norm{\nabla \mathbf{u}_t}_{L^2}^2
+\norm{\mathbf{u}_t}_{L^2\left(\partial B(0,b)\right)}^2}\\
 \leq &\epsilon
 \norm{\nabla \mathbf{u}_t}_{L^2}^2
 +\epsilon\norm{\mathbf{u}_t}_{L^2\left(\partial B(0,b)\right)}^2
 +D^{\epsilon}_0
  \norm{
 \nabla \mathbf{u}}^2_{L^2}
  \norm{
\mathbf{u}_t}^2_{L^2},
\end{aligned} \\ &\label{ut-proof-4}
\begin{aligned}
 \int_{\Omega}\abs{\rho ( \mathbf{u}\cdot\nabla)
(\nabla f\cdot \mathbf{u}_t)}
 \,dx\,dy\leq& C\norm{\nabla f}_{W^{1,\infty}}
   \norm{\rho}_{L^4} \norm{\mathbf{u}}_{L^4}
\norm{\nabla \mathbf{u}_t}_{L^2}\\
   &\leq \epsilon \norm{\nabla \mathbf{u}_t}^2_{L^2}
   +C_{\epsilon} \norm{\rho}^2_{L^4} 
   \norm{\mathbf{u}}_{L^2}\sqrt{
\norm{\nabla \mathbf{u}}_{L^2}^2
+\norm{\mathbf{u}}_{L^2\left(\partial B(0,b)\right)}^2}.
\end{aligned}
\end{align}

By choosing $\epsilon$ sufficiently small, we combine the results from \eqref{ut-proof-3} and \eqref{ut-proof-4} with \eqref{ut-proof-2} to derive the following conclusion:
  \begin{align}\label{ut-proof-5}
\begin{aligned}
&\frac{d}{dt}\norm{\mathbf{u}_t}^2_{L^2}
+\nu C_a^b \norm{\mathbf{u}_t}^2_{L^2}
\leq
D_{0}
\norm{\nabla \mathbf{u}}^2_{L^2}
  \norm{
\mathbf{u}_t}^2_{L^2}+
D_{0}  \left(
\norm{\nabla \mathbf{u}}_{L^2}^2
+\norm{\mathbf{u}}_{L^2\left(\partial B(0,b)\right)}^2\right)
\end{aligned}
\end{align}
which, combined with Gronwall's inequality, yields
  \begin{align}\label{ut-proof-6}
\begin{aligned}
\norm{\mathbf{u}_t}^2_{L^2}
\leq &\norm{\mathbf{u}_t(0)}^2_{L^2}e^{D_{0}\int_0^t
\left(
\norm{\nabla \mathbf{u}}_{L^2}^2
+\norm{\mathbf{u}}_{L^2\left(\partial B(0,b)\right)}^2\right)\,ds}
\\  &+D_{0} \norm{\rho_0}^2_{L^4} \int_0^t e^{D_{0}\int_s^t
\left(
\norm{\nabla \mathbf{u}}_{L^2}^2
+\norm{\mathbf{u}}_{L^2\left(\partial B(0,b)\right)}^2\right)\,ds}
\norm{\nabla \mathbf{u}(s)}^2_{L^2}\,ds.
\end{aligned}
\end{align}

To bound $\norm{\mathbf{u}_t(0)}^2_{L^2}$, we multiply the first equation of \eqref{213} with $\mathbf{u}_t$ in $L^2$ to get
\begin{align*}
\begin{aligned}
&\int_{\Omega}\rho^*\abs{\mathbf{u}_t}^2\,dx\,dy
-\mu\int_{\Omega}\Delta \mathbf{u}\cdot  \mathbf{u}_t\,dx\,dy
+\rho^*\int_{\Omega} \mathbf{u}_t\cdot( \mathbf{u}\cdot \nabla )\mathbf{u}\,dx\,dy
= -\int_{\Omega}\rho \nabla f\cdot \mathbf{u}_t \,dx\,dy,
\end{aligned}
\end{align*}
From this, it follows that
  \begin{align*}
\begin{aligned}
\norm{\mathbf{u}_t}^2_{L^2}
\leq \nu \norm{\mathbf{u}_t}_{L^2}
\norm{\Delta \mathbf{u}}_{L^2}
+ \norm{\mathbf{u}_t}_{L^2}
\norm{( \mathbf{u}\cdot \nabla )\mathbf{u}}_{L^2}
+\frac{\norm{\nabla f}_{L^{\infty}}}{\rho^*}
\norm{\rho_0}_{L^2}
\norm{\mathbf{u}_t}_{L^2},
\end{aligned}
\end{align*}
which gives
  \begin{align*}
\begin{aligned}
\norm{\mathbf{u}_t(0)}_{L^2}\leq
\nu\norm{\mathbf{u}_0}_{H^2}
+ \norm{\mathbf{u}_0}_{H^2}
\norm{\mathbf{u}_0}_{H^1}
+\frac{\norm{\nabla f}_{L^{\infty}}}{\rho^*}
\norm{\rho_0}_{L^2}.
\end{aligned}
\end{align*}
By the above inequality, \autoref{ll2} and \eqref{ut-proof-6}, one has
  \begin{align}\label{ut-proof-13}
\begin{aligned}
\mathbf{u}_t \in L^{\infty}\left((0,\infty),L^{2}(\Omega)\right).
\end{aligned}
\end{align}

By integrating \eqref{ut-proof-1} with respect to time, and using \autoref{lemma-grad-1-2} and \eqref{ut-proof-13} it follows that
\[
\mathbf{u}_t \in L^{\infty}\left((0,\infty);L^{2}(\Omega)\right)
\cap L^{2}\left((0,\infty);H^{1}(\Omega)\right).
\]
\end{proof}

\subsection{Estimate of $\norm{\mathbf{u}}_{ L^{\infty}\left((0,\infty);H^{2}(\Omega)\right)}$}

\begin{lemma}\label{21336}
For the problem \eqref{213}-\eqref{ii-c} subject to the boundary conditions \eqref{cond-1}-\eqref{cond-2} and
with initial data $(\mathbf{u}_0,\rho_0) \in H^{2}\left(\Omega\right)
 \times L^{4}\left(\Omega\right)$, if $\alpha \in  W^{1,\infty}\left(\partial B(0,a)\right)$,
 $\nabla f \in  W^{1,\infty}\left(\Omega\right)$ and $\left(\nu a^{-1}+\alpha\right)\geq 0$,
we get
\begin{align}
\mathbf{u} \in L^{\infty}\left((0,\infty),H^{2}(\Omega)\right).
\end{align}
\end{lemma}
\begin{proof}
Based on \autoref{lemma-grad-1-2} and \autoref{utestimate}, we conclude that both $\norm{\nabla \mathbf{u}}_{L^2}$ and $\norm{\mathbf{u}_t}_{L^2}$ are uniformly bounded in time. Then, appealing to \autoref{Stokes0725}, we find that
\begin{align}\label{H-2}
\begin{aligned}
\norm{ \mathbf{u}}_{H^2}^2
&\leq C\norm{\mathbf{u}_t}_{L^2}^2
+C\norm{\mathbf{u}}_{L^4}^4
\norm{\nabla \mathbf{u}}_{L^2}^2
\\&\quad+C\norm{\mathbf{u}}_{L^4}^2
\norm{\nabla \mathbf{u}}_{L^2}^2
+\norm{\rho  \nabla f}_{L^2}^2
\\&\leq 
C\norm{\mathbf{u}_t}_{L^2}^2
+C_0
\left(
\norm{\nabla \mathbf{u}}_{L^2}^2
+\norm{\mathbf{u}}_{L^2\left(\partial B(0,b)\right)}^2
\right)^2
\\&+C_0\left(\sqrt{
\norm{\nabla \mathbf{u}}_{L^2}^2
+\norm{\mathbf{u}}_{L^2\left(\partial B(0,b)\right)}^2}
\right)^3+\norm{\rho  \nabla f}_{L^2}^2<+\infty.
\end{aligned}
\end{align}
\end{proof}

\subsection{Estimate of $\norm{\mathbf{u}}_{ L^{2}\left((0,T);H^{3}(\Omega)\right)}$}

\begin{lemma}\label{lemma-37}
For the problem \eqref{213}-\eqref{ii-c} subject to the boundary conditions \eqref{cond-1}-\eqref{cond-2} and with initial data $(\mathbf{u}_0,\rho_0) \in H^{3}\left(\Omega\right)
 \times W^{1,q}\left(\Omega\right)$, If $\alpha \in  W^{1,\infty}\left(\partial B(0,a)\right)$,
 $\nabla f \in  W^{1,\infty}\left(\Omega\right)$ and $\left(\nu a^{-1}+\alpha\right)\geq 0$,
then for any $T>0$, we have
\begin{align}
\mathbf{u} \in L^{2}\left((0,T);H^{3}(\Omega)\right),\quad
 \rho \in L^{\infty}\left((0,T);W^{1,q}(\Omega)\right).
\end{align}
\end{lemma}
\begin{proof}
We replace $\rho$ and $P$ by $\rho-\rho_s$ and $P-P_s$ in the system \eqref{213}, where $\rho_s$ and $P_s$ solve
$\nabla P_s=-\rho_s\nabla f$.  Let us denote $\omega=\nabla^{\perp}\cdot\mathbf{u}=( -\partial_y,\partial_x)\cdot\mathbf{u}$.
After a straightforward calculation, it can be seen tha $\omega$ solves
   \begin{align}\label{three-proof--3}
\begin{cases}
\frac{\partial \omega}{\partial t}+ (\mathbf{u} \cdot \nabla )\omega = \nu \Delta \omega
+\frac{1}{\rho^*}\nabla\rho \cdot \nabla^{\perp}f ,\\
\omega=0,\quad x^2+y^2=b^2,\\
\omega=-\phi u_{\tau}. \quad x^2+y^2=a^2,\\
\omega|_{t=0}= \omega_0 ,
\end{cases}
\end{align}
where $\phi $ is given by $\phi:=2a^{-1}+\frac{\alpha}{\nu}$,
 and $\eqref{three-proof--3}_2-\eqref{three-proof--3}_3$ are derived from \eqref{cond-1} and
\begin{align*}
\omega=\partial_xu_2-\partial_yu_1=
\frac{\partial u_{\tau}}{\partial \mathbf{n}}-\frac{u_{\tau}}{r}
+\frac{1}{r}\frac{\partial u_{\mathbf{n} }}{\partial \tau}
\end{align*}
with $\mathbf{n}$ and $\tau$ on $x^2+y^2=a^2$ being given by $
\mathbf{n}=-(\cos\theta,\sin\theta),\quad \tau
=(\sin\theta,-\cos\theta)$.

If we multiply the corresponding \eqref{three-proof--3} by $\Delta \omega$ in $L^2$, we obtain
\begin{align}\label{proof-37-1}
\begin{aligned}
\int_{\Omega}\omega_t \Delta \omega\,dx\,dy
+
\int_{\Omega} \Delta \omega (\mathbf{u} \cdot \nabla )\omega \,dx\,dy=
 \nu  \int_{\Omega} \left(\Delta \omega\right)^2 \,dx\,dy  +\frac{1}{\rho^*}
  \int_{\Omega} \Delta \omega\left(\nabla\rho \cdot \nabla^{\perp}f \right) \,dx\,dy.
\end{aligned}
\end{align}
The first term on the left hand side of \eqref{proof-37-1} can be rewritten as
\begin{align}\label{proof-37-2}
\begin{aligned}
\int_{\Omega}\omega_t
\Delta \omega\,dx\,dy
=-\int_{\Omega}\nabla \omega_t
\cdot\nabla \omega\,dx\,dy+
\int_{\partial\Omega}\omega_t \mathbf{n}
\cdot\nabla \omega\,dx\,dy,
\end{aligned}
\end{align}
which combined with \eqref{proof-37-1} implies
\begin{align}\label{proof-37-3}
\begin{aligned}
\frac{1}{2}\frac{d}{dt}\norm{\nabla \omega}_{L^2}^2
+\nu \norm{\Delta \omega}_{L^2}^2
=&
\int_{\partial\Omega}\omega_t \mathbf{n}
\cdot\nabla \omega\,dx\,dy
+
\int_{\Omega}
\Delta \omega
 (\mathbf{u} \cdot \nabla )\omega
 \,dx\,dy\\&
 -\frac{1}{\rho^*}
  \int_{\Omega}
\Delta \omega\left(\nabla\rho \cdot \nabla^{\perp}f \right)
 \,dx\,dy.
\end{aligned}
\end{align}

Keeping in mind that
\begin{align*}
\begin{aligned}
&\omega_t=0, \quad x^2+y^2=b^2;\quad \omega_t=-\phi \mathbf{u_{t}}\cdot \tau, \quad x^2+y^2=a^2,
\end{aligned}
\end{align*}
we infer after a straightforward calculation that
\begin{align}\label{proof-37-5}
\begin{aligned}
\frac{1}{2}\frac{d}{dt}\norm{\nabla \omega}_{L^2}^2 +\nu \norm{\Delta \omega}_{L^2}^2
= & -
\int_{\partial B(0,a)} (\phi \mathbf{u_{t}}\cdot \tau) \mathbf{n}
\cdot\nabla \omega\,dx\,dy
+
\int_{\Omega} \Delta \omega (\mathbf{u} \cdot \nabla )\omega  \,dx\,dy
 \\  &   -\frac{1}{\rho^*} \int_{\Omega}
\Delta \omega\left(\nabla\rho \cdot \nabla^{\perp}f \right) \,dx\,dy=L_1+L_2+L_3,
\end{aligned}
\end{align}
where the terms $L_j$ can be bounded as follows.
\begin{align}\label{proof-37-6}
\begin{aligned}
L_1&\leq C\norm{\phi}_{W^{1,\infty}}
\norm{(\mathbf{u_{t}}\cdot \tau)
\mathbf{n}
\cdot\nabla \omega}_{W^{1,1}}
\\&\leq C\norm{\phi}_{W^{1,\infty}}\norm{\mathbf{u_{t}}}_{H^{1}}
\norm{\nabla \omega }_{H^{1}}\leq \epsilon \norm{\nabla \omega }_{H^{1}}^2+
C_{\epsilon }\norm{\mathbf{u_{t}}}_{H^{1}}^2,
\end{aligned}
\end{align}
while by virtue of Holder's and Young's inequalities, and Ladyzhenskaya's interpolation inequality,
\begin{align}\label{proof-37-7}
\begin{aligned}
L_2&\leq
\norm{\mathbf{u}}_{L^4}
\norm{\nabla \omega}_{L^4}
\norm{\Delta \omega}_{L^2}
\leq C\norm{\mathbf{u}}_{H^1}
\norm{\nabla \omega}_{H^1}^{\frac{1}{2}}
\norm{\nabla \omega}_{L^2}^
{\frac{1}{2}}\norm{\Delta \omega}_{L^2}\\
&\leq
C\norm{\mathbf{u}}_{H^1}
\norm{\nabla \omega}_{H^1}^{\frac{3}{2}}
\norm{\nabla \omega}_{L^2}^
{\frac{1}{2}}\leq
 \epsilon \norm{\nabla \omega }_{H^{1}}^2
+C_{\epsilon}\norm{\mathbf{u}}_{H^1}^4
\norm{\nabla \omega}_{L^2}^2.
\end{aligned}
\end{align}

Recalling that
\[
 \norm{\nabla \omega }_{H^{1}}^2\leq C\left(
  \norm{\Delta \omega }_{L^{2}}^2
  +\norm{\mathbf{u}}_{H^{2}}^2
 \right),
\]
we see that
\begin{align}\label{proof-37-9}
\begin{aligned}
L_1+L_2&\leq  \epsilon C \norm{\Delta \omega }_{L^{2}}^2
+
C_{\epsilon }C\norm{\mathbf{u}}_{H^{2}}^2
+
C_{\epsilon }\norm{\mathbf{u_{t}}}_{H^{1}}^2
+C_{\epsilon}\norm{\mathbf{u}}_{H^1}^4
\norm{\mathbf{u}}_{H^2}^2.
\end{aligned}
\end{align}
As for $L_3$, we have
\begin{align}\label{proof-37-10}
\begin{aligned}
L_3&\leq
\norm{f}_{L^{\infty}}
\norm{\Delta \omega}_{L^{2}}
\norm{\nabla \rho}_{L^{2}}\leq
\epsilon
\norm{\Delta \omega}_{L^{2}}^2+
C_{\epsilon}\norm{\nabla \rho}_{L^{2}}^2\\
\leq &
\epsilon
\norm{\Delta \omega}_{L^{2}}^2
+
C_{\epsilon}
+C_{\epsilon}\norm{ \nabla \rho}_{L^{q}}^q,\quad q\geq 2.
\end{aligned}
\end{align}

Combining \eqref{proof-37-5} with \eqref{proof-37-9}-\eqref{proof-37-10}, and choosing $\epsilon$ sufficiently small, we conclude
\begin{align}\label{proof-37-12}
\begin{aligned}
\frac{1}{2}\frac{d}{dt}\norm{\nabla \omega}_{L^2}^2
+\frac{\nu}{2} \norm{\Delta \omega}_{L^2}^2
&\leq
C_{\epsilon }\norm{\mathbf{u}}_{H^{2}}^2
+
C_{\epsilon }\norm{\mathbf{u_{t}}}_{H^{1}}^2
\\&\quad+C_{\epsilon}\norm{\mathbf{u}}_{H^1}^4
\norm{\mathbf{u}}_{H^2}^2
+C_{\epsilon}
+C_{\epsilon}\norm{\nabla \rho}_{L^{q}}^q.
\end{aligned}
\end{align}

Now, we take the gradient of the second equation of \eqref{213} with $\rho$ replaced by $\rho-\rho_s$,
and multiply $\abs{\nabla \rho}^{q-2}\nabla \rho$ to find that
\begin{align}\label{proof-37-13}
\begin{aligned}
\frac{1}{q}\frac{d}{dt}\norm{\nabla \rho}_q^q
=&-\int_{\Omega}\abs{\nabla \rho}^{q-2}\nabla \rho \cdot
(\nabla \rho\cdot\nabla)\mathbf{u}\,dx\,dy-\int_{\Omega}\abs{\nabla \rho}^{q-2}\nabla \rho \cdot
(\mathbf{u} \cdot\nabla)\nabla \rho\,dx
\,dy\\&=
-\int_{\Omega}\abs{\nabla \rho}^{q-2}\nabla \rho \cdot
(\nabla \rho\cdot\nabla)\mathbf{u}\,dx\,dy
\\&\leq \norm{\nabla\mathbf{u}}_{L^{\infty}}
\norm{\nabla \rho}_{L^{q}}^q
\leq C
\left(1+
\norm{\nabla\mathbf{u}}_{H^{1}}\right)
\log (1+\norm{\nabla\mathbf{u}}_{H^{2}})
\norm{\nabla \rho}_{L^{q}}^q\\
&\leq C
\left(1+
\norm{\mathbf{u}}_{H^{2}}\right)
\norm{\nabla \rho}_{L^{q}}^q
+\left(1+
\norm{\mathbf{u}}_{H^{2}}\right)
\log (1+\norm{\Delta \omega}_{L^{2}})
\norm{\nabla \rho}_{L^{q}}^q,
\end{aligned}
\end{align}
where we have used the logarithmic Sobolev inequality
$ \norm{\nabla\mathbf{u}}_{L^{\infty}}\leq
C\left(1+
\norm{\nabla\mathbf{u}}_{H^{1}}\right)
\log (1+\norm{\nabla\mathbf{u}}_{H^{2}})$ and
$ \norm{\nabla \mathbf{u} }_{H^{2}}^2\leq C( \norm{\Delta \omega }_{L^{2}}^2 +\norm{\mathbf{u}}_{H^{2}}^2)$.

Finally, combining \eqref{proof-37-5} and \eqref{proof-37-13}, we get
\begin{align}\label{proof-37-14}
\begin{aligned}
&\frac{d}{dt}\norm{\nabla \omega}_{L^2}^2 +\nu\norm{\Delta \omega}_{L^2}^2
+\frac{2}{q}\frac{d}{dt}\norm{\nabla \rho}_q^q \leq C_{\epsilon }\norm{\mathbf{u}}_{H^{2}}^2 +
C_{\epsilon }\norm{\mathbf{u_{t}}}_{H^{1}}^2
\\& \quad +C_{\epsilon}\norm{\mathbf{u}}_{H^1}^4 \norm{\mathbf{u}}_{H^2}^2 +C_{\epsilon}
+C_{\epsilon}\norm{ \rho}_{L^{q}}^q +C
\left( 1+ \norm{\mathbf{u}}_{H^{2}}\right) \norm{\nabla \rho}_{L^{q}}^q
\\
&\quad+\left( 1+ \norm{\mathbf{u}}_{H^{2}}\right) \log (1+\norm{\Delta \omega}_{L^{2}}) \norm{\nabla \rho}_{L^{q}}^q.
\end{aligned}
\end{align}
Denoting
\begin{align}\label{proof-37-141}
\begin{aligned}
&X(t)=\norm{\nabla \omega}_{L^2}^2
+\frac{2}{q}\norm{\nabla \rho}_q^q+C_{\epsilon},\quad
 Y(t)=\nu\norm{\Delta \omega}_{L^2}^2,\\
 & A(t)=  C_{\epsilon }\left(\norm{\mathbf{u}}_{H^{2}}^2
+ \norm{\mathbf{u_{t}}}_{H^{1}}^2 +\norm{\mathbf{u}}_{H^1}^4 \norm{\mathbf{u}}_{H^2}^2 +1\right),\quad B(t)=C \left( 1 + \norm{\mathbf{u}}_{H^{2}}\right),
\end{aligned}
\end{align}
we employ the logarithmic Gronwall inequality and \autoref{lemma-grad-3} to deduce
\begin{align}\label{proof-37-15}
\begin{aligned}
&\nabla \rho \in L^{\infty}\left((0,T);L^q(\Omega)\right),\;\;\quad
\Delta \omega\in L^{2}\left((0,T);L^2(\Omega)\right).
\end{aligned}
\end{align}
Hence, by \eqref{rhopp},  Lemma \eqref{21336} and \eqref{proof-37-15}, one has
\begin{align}\label{proof-37-16}
\begin{aligned}
& \rho \in L^{\infty}\left((0,T);W^{1,q}(\Omega)\right),\;\;\quad \mathbf{u}\in L^{2}\left((0,T);H^3(\Omega)\right).
\end{aligned}
\end{align}
\end{proof}

\section{Proofs of the main theorems}
\subsection{Proof of \autoref{theorem-3}}\label{section-1}
    \subsubsection{Eigenvalue problem}
    The eigenvalue problem for the linear equations \eqref{213-2} reads as
    \begin{align}\label{three-proof--eigen-3}
\begin{cases}
\displaystyle{ \lambda \mathbf{u}= \nu \Delta \mathbf{u}-
\frac{1}{\rho^*} \nabla P-\frac{\rho}{\rho^*} \nabla f, }
 \\[2mm]  \lambda \rho=- h(x,y) (\mathbf{u}\cdot  \nabla)f ,
\\ \nabla \cdot  \mathbf{u}=0,\\[2mm]
\displaystyle{ \frac{\partial u_{\mathbf{\tau}} }{\partial\mathbf{n}}
  =-\left(a^{-1}+\frac{\alpha}{\nu}\right)u_{\mathbf{\tau}},\quad x^2+y^2=a^2,}  \\[2mm]
 \mathbf{u}\cdot  \mathbf{n}=0,\quad x^2+y^2=a^2,\\
 \mathbf{u} \cdot   \mathbf{n}|_{x^2+y^2=b^2}=\nabla \times   \mathbf{u}|_{x^2+y^2=b^2}=0,.
\end{cases}
\end{align}
By multiplying the first equation of the problem \eqref{three-proof--eigen-3} by $\lambda$ and eliminating the density $\rho$, one subsequently obtains
    \begin{align}\label{three-proof--eigen-4-1-1}
\begin{cases} \displaystyle{
-\lambda^2 \mathbf{u}= -\nu \lambda \Delta \mathbf{u} +
\frac{\lambda}{\rho^*} \nabla P-\frac{h(x,y) }{\rho^*} (\mathbf{u}\cdot  \nabla f) \nabla f,} \\
\nabla \cdot  \mathbf{u}=0,\\[1mm]
\displaystyle{ \frac{\partial u_{\mathbf{\tau}} }{\partial\mathbf{n}}
  =-\left(a^{-1}+\frac{\alpha}{\nu}\right)u_{\mathbf{\tau}},\quad x^2+y^2=a^2,} \\
 \mathbf{u}\cdot  \mathbf{n}=0,\quad x^2+y^2=a^2,\\
\mathbf{u} \cdot   \mathbf{n}|_{x^2+y^2=b^2}=\nabla \times   \mathbf{u}|_{x^2+y^2=b^2}=0.
\end{cases}
\end{align}

Please note that if the problem \eqref{three-proof--eigen-4-1-1} has a nonzero eigenvalue 
$\lambda$, let $\rho$ be given by $\rho=-\lambda^{-1} h(x,y)  (\mathbf{u}\cdot  \nabla)f $,
then $(\lambda, \mathbf{u},\rho)$ solves the problem \eqref{three-proof--eigen-3}.
Since the terms containing $\lambda$ breaks down the variational structure of the problem \eqref{three-proof--eigen-4-1-1}, it prevents the direct application of the standard variational
method for establishing the existence of solutions
to \eqref{three-proof--eigen-4-1-1}.  Replacing $\lambda$ by $s$ on right hand side of
$ \eqref{three-proof--eigen-4-1-1}_1$, we consider the following modified eigenvalue problem:
    \begin{align}\label{three-proof--eigen-4-1}
\begin{cases}   \displaystyle{
-\lambda^2 \mathbf{u}= -\nu s \Delta \mathbf{u} +
\frac{s}{\rho^*} \nabla P-\frac{h(x,y) }{\rho^*} (\mathbf{u}\cdot\nabla f) \nabla f,} \\[1mm]
\nabla \cdot  \mathbf{u}=0,\\[2mm]
 \displaystyle{ \frac{\partial u_{\mathbf{\tau}} }{\partial\mathbf{n}}
  =-\left(a^{-1}+\frac{\alpha}{\nu}\right)u_{\mathbf{\tau}},\quad x^2+y^2=a^2,}\\[1mm]
 \mathbf{u}\cdot  \mathbf{n}=0,\quad x^2+y^2=a^2,\\
 \mathbf{u} \cdot   \mathbf{n}|_{x^2+y^2=b^2}=\nabla \times   \mathbf{u}|_{x^2+y^2=b^2}=0.
\end{cases}
\end{align}

Denoting the functionals $J$, $E$ and $E_k$ ($k=1,2$) by
\begin{subequations}
\begin{align}
    &J( \mathbf{u})=\int_{\Omega} \abs{ \mathbf{u}}^2\,dx\,dy,\quad E( s, \mathbf{u})=s\nu E_1( \mathbf{u})-\frac{1}{\rho^*} E_2( \mathbf{u}),\\
    &E_1( \mathbf{u})=
     \int_{\Omega}\abs{\nabla\mathbf{u}}^2\,dx\,dy
     +\red{\frac{1}{b} \int_{\partial B(0,b)}\left(u_{\mathbf{\tau}}\right)^2 \,ds}
   + \int_{\partial B(0,a)}\left(a^{-1}+\frac{\alpha}{\nu}\right)\left(u_{\mathbf{\tau}}\right)^2 \,ds,\\
   &E_2( \mathbf{u})=
     \int_{\Omega}h(x,y)(\mathbf{u}\cdot  \nabla f)^2
     \,dx\,dy ,
\end{align}
\end{subequations}
we consider the following variational problem:
\begin{align}\label{variational-problem}
-\lambda^2=\inf_{ \mathbf{u}\in \mathcal{A}}E( s, \mathbf{u}),
\end{align}
where $\mathcal{A}$ is given by
\begin{align}
\begin{aligned}
 &\mathcal{A}=\left\{ \mathbf{u}\in  \mathcal{X}\, |\, J( \mathbf{u})=1\right\},\\
& \mathcal{X}=
 \left\{ \mathbf{u}\in H^1(\Omega)| \nabla \cdot \mathbf{u}=0,~ \mathbf{u}~\text{satisfies}~\eqref{cond-1} \right\}.
 \end{aligned}
\end{align}

By virtue of the assumption $h(x,y)>0$ at $(x,y)=(x_0,y_0)\in\Omega$, there exists an $\mathbf{u}\in \mathcal{X}$,such that $E_2( \mathbf{u})>0$. Therefore, let $s$ be any fixed positive number, satisfying
\begin{align}\label{ssss}
\begin{aligned}
s^{-1}>\min_{\substack{\mathbf{u}\in  \mathcal{X}\\E_2( \mathbf{u})=1}}
\rho^*\nu E_1( \mathbf{u})=s_0^{-1}.
 \end{aligned}
\end{align}
Noting that \eqref{ssss} means
\[
-(\lambda(s))^2=\inf_{ \mathbf{u}\in \mathcal{A}}E( s, \mathbf{u})<0,
\]
one further has the following results
\begin{proposition}\label{pro-1}
If $\rho_s,f \in C^{1}(\overline{\Omega})$, $h(x_0,y_0)>0$ at $(x_0,y_0)\in\Omega$,
 and $s$ satisfying \eqref{ssss}, then $E( s, \mathbf{u})$ achieves its minimum on $ \mathcal{A}$.
Moreover, if $\mathbf{u}$ solves the equations \eqref{variational-problem},
then there is a pressure field $p$ associated to $\mathbf{u}$, such that $(\mathbf{u},p,\lambda)$ solves
\eqref{three-proof--eigen-4-1}
and $(\mathbf{u},p) \in H^2(\Omega)\times H^1(\Omega)$.
\end{proposition}
\begin{proof}
According to the definition of $E( s, \mathbf{u})$, $J(\mathbf{u})=1$, we have
\begin{align}\label{cccc}
E( s, \mathbf{u})=s\nu E_1( \mathbf{u})-\frac{1}{\rho^*} E_2( \mathbf{u})
>\frac{1}{\rho^*} E_2( \mathbf{u})>-\frac{
\norm{h(x,y)}_{L^{\infty}}\norm{\nabla f}^2_{L^{\infty}}}{\rho^*} .
\end{align}

In order to show $E(s,\mathbf{u})$ achieves its minimum on $\mathcal{A}$, we denote by
$\{ \mathbf{u}_{n}\}_{n = 1}^{ \infty} \in \mathcal{A}$ a minimizing sequence of the energy $E(s,\mathbf{u})$.
Thus, we assume that $E(s,\mathbf{u}_n)$ is bounded, and denote the bound by $B = \sup_n E(s,\mathbf{u}_n)$. Consequently,
 $\{\mathbf{u}_{n}\}_{n = 1}^{\infty}$ is a bounded sequence in $H^1(\Omega)$. Thus, there is a convergent subsequence $\{ \mathbf{u}_{n'}\} $,
 such that $\mathbf{u}_{n'}\to\mathbf{u}$ weakly in $H^1(\Omega)$ and strongly in $L^2(\Omega)$. Therefore,
 $J( \mathbf{u})=1$ remains true since it is a strong continuous functional on $L^2(\Omega)$, and hence, $\mathbf{u}\in \mathcal{A}$.
 Moreover, by the weak lower semicontinuity of $E_1$ and strong continuity of  $E_2$, one can show the weak lower semicontinuity
 of $E( s,\mathbf{u})$. Therefore,
    \begin{align}
        E(s,\mathbf{u}) \leq \liminf_{n'} E( s,\mathbf{u}_{n'} ) =  \inf_{\mathbf{u}\in\mathcal{A}}E( s,\mathbf{u} ).
    \end{align}

Finally, we show that the minimizer $\tilde{\mathbf{u}}$ solves \eqref{three-proof--eigen-4-1}.
For any $C_0^{ \infty}(\Omega)$ test function $\mathbf{w} = (w_1,w_2)$, we define
    \begin{align*}
        I(t) := \frac{E(s,\mathbf{u}+ t \mathbf{w} )}
        {J(\mathbf{u} + t \mathbf{w})} .
    \end{align*}
    Since $\tilde{\mathbf{u}}$ is the minimizer, $I'(0) = 0$ is valid. A straightforward calculation shows that
    \begin{align*}
        \begin{aligned}
            I'(0) & = \frac{1}{J(\mathbf{u})}  \left( DE(s,\mathbf{u} )  \mathbf{w}+ \lambda^2 D J(\mathbf{u} )  \mathbf{w} \right),
        \end{aligned}
    \end{align*}
    where we have used the fact $\frac{E(s,\mathbf{u})}{J(\mathbf{u})} =- \lambda^2$.
    Keeping in mind that $J(\mathbf{u}) > 0$, one gets
        \begin{align*}
        \begin{aligned}
&s\nu \int_{\Omega}\nabla \mathbf{u} \cdot \nabla\mathbf{w}\,dx\,dy
+\red{\frac{s\nu}{b} \int_{\partial B(0,b)}u_{\tau} w_{\tau}\ \,ds}
   +s\nu \int_{\partial B(0,a)}\left(a^{-1}+\frac{\alpha}{\nu}\right)u_{\tau} w_{\tau}\,ds
   \\ &-\frac{1}{\rho^*} \int_{\Omega}h(x,y)(\mathbf{u}\cdot\nabla f)(\mathbf{w}\cdot\nabla f)\,dx\,dy=
   -\lambda^2\int_{\Omega} \mathbf{u} \cdot \mathbf{w}\,dx\,dy .
           \end{aligned}
    \end{align*}
which means that $\mathbf{u}$ is a weak solution to the boundary problem \eqref{three-proof--eigen-4-1}.
Thus, it follows from the regularity theory (see \autoref{Stokes0725}) on the Stokes equations
that there are constants $c_1$ dependent on the domain $\Omega$, and $ c_2$, such that
\begin{align*}
    \norm{\mathbf{u}}_{H^2} + \norm{\nabla p}_{L^2} \leq \frac{c_1}{s \nu }\norm{\frac{\gamma}{\rho^*}
 (\mathbf{u}\cdot  \nabla f)
 \nabla f +\lambda^2 \mathbf{u}}_{L^2} \leq c_2,
\end{align*}
where $p \in H^{1}(\Omega)$ is the pressure field. This immediately gives that $\mathbf{u}\in H^{2}(\Omega) $, $p \in H^{1}(\Omega)$.
\end{proof}

\begin{proposition}\label{properties_of_alpha}
 Under the conditions of \autoref{pro-1}, the function 
 \[
 \alpha(s):=\inf_{ \mathbf{u}\in \mathcal{A}}E( s, \mathbf{u})
 \]
  defined on $( 0, s_0 )$ enjoys the following properties:
    \begin{enumerate}
        \item[\rm{(1)}] For any $ s\in(0, s_0)$, there exist positive numbers $C_i$ ($i = 1,2$), depending on $( f,\nu,\rho^*,\gamma)$,
        such that
              \begin{equation}
                  \alpha(s ) \leq  - C_1 + s C_2;
              \end{equation}
        \item[\rm{(2)}]  $ \alpha(s) \in C^{0,1}_{loc}(0, s_0)\) is
        continuous and
        strictly increasing.
    \end{enumerate}
\end{proposition}
\begin{proof}
There exists $\tilde{\mathbf{u}}$, such that $J(\tilde{\mathbf{u}})=1$. Then,
\[
\alpha(s)\leq -C_2+sC_1,\quad
C_1=\nu E_1( \tilde{\mathbf{u}}),\quad C_2=-\frac{1}{\rho^*} E_2(\tilde{\mathbf{u}}).
\]
Note that $E( s_1, \mathbf{u})<E( s_2, \mathbf{u})$ with $s_1<s_2$
and subjected with $J(\mathbf{u})=1$. Hence,
\begin{align}\label{a12-1}
\begin{aligned}
\alpha(s_2)-
\alpha(s_1)&=
E( s_2, \mathbf{u}_2)- E( s_1, \mathbf{u}_1)\\&=
E( s_2, \mathbf{u}_2)-E( s_1, \mathbf{u}_2)
+E( s_1, \mathbf{u}_2)-E( s_1, \mathbf{u}_1)\\
&\geq \nu(s_2-s_1)E_1( \mathbf{u}_2)>0.
\end{aligned}
\end{align}
In a similar way, we have
\begin{align}\label{a12-2}
\begin{aligned}
\alpha(s_2)-
\alpha(s_1)&=
E( s_2, \mathbf{u}_2)- E( s_1, \mathbf{u}_1)\\&=
E( s_2, \mathbf{u}_2)-E( s_2, \mathbf{u}_1)
+E( s_2, \mathbf{u}_1)-E( s_1, \mathbf{u}_1)\leq \nu(s_2-s_1)E_1( \mathbf{u}_1).
\end{aligned}
\end{align}
Hence, from \eqref{a12-1} and \eqref{a12-2} it follows that
$ \alpha(s) \in C^{0,1}_{\mathrm{loc}}(0, s_0)\) is
        continuous and
        strictly increasing.
\end{proof}

By virtue of \autoref{properties_of_alpha}, there exists a constant $ s^* > 0$ depending on the quantities $( f,\nu,\rho^*,\gamma)$,
such that for any $s \leq s^*$, $\alpha(s ) < 0$. Denoting
\begin{equation}
    s_c:=\sup \{ s\ |\ \alpha(t ) < 0 \quad
    \text{for any $t\in (0,s)$}  \} ,
\end{equation}
one thus can define $ \lambda(s ) = \sqrt{ - \alpha(s )} > 0$ for any $s \in \mathcal{S} :=(0,\mathcal{G})$.

Now, we are able to prove the existence of the fixed point:
\begin{proposition}\label{proposition-3}
    There exists a $s^* \in \mathcal{S}$, such that
    \begin{align*}
        -(s^*)^2=-\lambda^2(s^*) = \alpha(\lambda( s^* )) = \inf_{\mathbf{u} \in \mathcal{A}} E(s,\mathbf{u}).
    \end{align*}
\end{proposition}
\begin{proof}
    Define an auxiliary function $
        \Phi(s):= - s^2 - \alpha(s)$.
    According to (1) of \autoref{properties_of_alpha} and the lower boundedness of $E(s,\mathbf{u})$
    (see \eqref{cccc}), one has
    \begin{equation}
        \begin{aligned}
            \lim_{s \to 0^{ +}}\Phi(s) \geq  \lim_{s \to 0^{ +}}- s^2 +C_1 - s C_2 = C_1 > 0, \quad \lim_{s \to s_c^{-} }\Phi(s) \leq 0.
        \end{aligned}
    \end{equation}
    Hence, by the continuity of $\alpha(s)$, one can obtain some $s^* > 0$,
     such that $ -\left(s^*\right)2 - \alpha(s^*) = 0$. Thus, $\alpha(\lambda) < 0$, and therefore $s^*\in \mathcal{S}$,
     which completes the proof.
\end{proof}

\begin{lemma}\label{linearins}
Suppose that $f \in C^{1}(\overline{\Omega})$. If the profile $\rho_s \in C^{1}(\overline{\Omega}) $ satisfying
$\nabla \rho_s =h(x,y) \nabla f$ where $h(x_0,y_0)>0$ at $(x_0,y_0)\in \Omega$ ,
then the eigenvalue problem \eqref{three-proof--eigen-3} has
a positive eigenvalue. For $h(x,y)<h_0< 0$, all eigenvalues $\lambda$
of the prolem \eqref{three-proof--eigen-3}
 are negative.
\end{lemma}
\begin{proof}
The first part is derived from  \autoref{pro-1} and \autoref{proposition-3}. The second part can be established by observing that
\begin{align*}
\lambda\left(\rho^*\norm{\mathbf{u}}_{L^2}^2+
\norm{\frac{\rho}{-h(x,y)}}^2_{L^2}\right)=&-\rho^*\nu
  \int_{\Omega}\abs{\nabla\mathbf{u}}^2\,dx\,dy
   -\red{\frac{ \rho^*\nu }{b}
   \int_{\partial B(0,b)}\left(u_{\mathbf{\tau}}\right)^2 \,ds}\\&
   -\rho^*\nu \int_{\partial B(0,a)}\left(a^{-1}+\frac{\alpha}{\nu}\right)\left(u_{\mathbf{\tau}}\right)^2 \,ds
   <0.
\end{align*}
\end{proof}
Finally, we obtain \autoref{theorem-3} by applying \autoref{linearins}.

\subsection{Proof of \autoref{theorem-1}}
Utilizing Gagliardo-Nirenberg interpolation, we have
\[
\norm{\mathbf{u}}_{W^{1,p}}^p\leq \norm{\mathbf{u}}_{H^1}^{2}
\norm{\mathbf{u}}_{H^2}^{p-2},\quad 2\leq p.
\]
From \autoref{lemma-grad} and \autoref{21336}, and the above inequality, \eqref{theorem-1-conc-1} follows.
 \autoref{utestimate} yields \eqref{theorem-1-conc-2}. The conclusion \eqref{theorem-1-conc-3} can be obtained by \autoref{Stokes0725},
  \autoref{utestimate} and \autoref{21336},
while \eqref{theorem-1-conc-4} is derived from \eqref{rhopp}.
When $(\mathbf{u}_0,\rho_0) \in H^{3}\left(\Omega\right) \times W^{1,q}\left(\Omega\right)$,
\eqref {theorem-1-conc-12} follows from \autoref{lemma-37}.
By Gagliardo-Nirenberg interpolation and \autoref{21336}, we infer that
\[
\begin{cases}
\int_{0}^{T}\norm{\mathbf{u}}^{\frac{2p}{p-2}}_{W^{2,p}}\,dt
\leq C
\norm{\mathbf{u}}_{ L^{\infty}\left((0,\infty);H^{2}(\Omega)\right)}^{\frac{4}{p-2}}
\int_{0}^{T}\norm{\mathbf{u}}^{2}_{H^{3}}\,dt<\infty,\quad 2<p,\\[1mm]
\mathbf{u} \in L^{\infty}\left((0,\infty),H^{2}(\Omega)\right),\quad p=2,
\end{cases}
\]
which gives \eqref {theorem-1-conc-11}.

\subsection{Proof of \autoref{theorem-2}}

\subsubsection{Large time behavior
of $\norm{\mathbf{u}}_{H^{1}}$}

Let us denote 
\[
g(t):=\nu\int_\Omega |\nabla \mathbf{u}|^2\,dx\,dy
+\frac{\nu}{b}\int_{\partial B(0,b)}\left(u_{\mathbf{\tau}}\right)^2 \,ds
+\int_{\partial B(0,a)}
\left(\nu a^{-1}+\alpha\right)\left(u_{\mathbf{\tau}}\right)^2 \,ds.
\]
One can get from \eqref{one-proof-1-n} that
$g(t)\in L^1(0,+\infty)$.
Using the method deriving \eqref{new-asymptotical}, we can get
\begin{align}\label{new-asymptotical-1}
\begin{aligned}
&
\nu\rho^*\frac{1}{2} \frac{d}{dt} \int_\Omega |\nabla \mathbf{u}|^2\,dx\,dy
+\frac{\rho^*\nu}{2b} \frac{d}{dt}
\int_{\partial B(0,b)}\left(u_{\mathbf{\tau}}\right)^2 \,ds\\&
+\frac{\rho^*}{2} \frac{d}{dt}
\int_{\partial B(0,a)}
\left(\nu a^{-1}+\alpha\right)\left(u_{\mathbf{\tau}}\right)^2 \,ds
+ \frac{\rho^*}{2} \|\mathbf{u}_t\|_{L^2}^2 \\ \le  &
C_0
\left(
\norm{\nabla \mathbf{u}}_{L^2}^2
+\norm{\mathbf{u}}_{L^2\left(\partial B(0,b)\right)}^2
\right)^2+C_0\norm{\rho}_{L^2}^2\\
&+C_0\left(
\sqrt{
\norm{\nabla \mathbf{u}}_{L^2}^2
+\norm{\mathbf{u}}_{L^2\left(\partial B(0,b)\right)}^2}
\right)^3<C_0.
\end{aligned}
\end{align}
This shows that $g(t)$ is a uniformly continuous function.  Then, 
\autoref{lemma-grad-2} implies that
\[
\int_\Omega |\nabla \mathbf{u}|^2\,dx\,dy
+\int_{\partial B(0,b)}\left(u_{\mathbf{\tau}}\right)^2 \,ds \to 0\quad \text{as}\quad t\to +\infty.
\]
An application of \autoref{lemma-grad}, we have $
\int_\Omega |\mathbf{u}|^2\,dx\,dy \to 0\quad \text{as}\quad t\to +\infty$.
Collecting these results shows $
\norm{\mathbf{u}}_{H^1} \to 0\quad \text{as}\quad t\to +\infty$. Using Gagliardo-Nirenberg's interpolation , we have
\[
\begin{cases}
\norm{\mathbf{u}}_{W^{1,p}}\leq \norm{\mathbf{u}}_{H^1}^{\frac{2}{p}}
\norm{\mathbf{u}}_{H^2}^{1-\frac{2}{p}},\quad 2<p,\\
\norm{\mathbf{u}}_{W^{1,p}}\leq C\norm{\mathbf{u}}_{H^1}^p,
1\leq p\leq 2.
\end{cases}
\]
Since $\norm{\mathbf{u}}_{H^2}$ is uniformly bounded in time (see \autoref{21336}), we find that $
\norm{\mathbf{u}}_{W^{1,p}}\to 0\quad\text{as}\quad t\to \infty$,
which gives \eqref{theorem-2-conc-1}.

\subsubsection{Large time behavior of $\norm{\mathbf{u}_t}_{L^{2}}$}
First, \autoref{utestimate} shows that $\norm{\mathbf{u}_t}_{L^{2}}^2\in L^1(0,+\infty)$.
Second, from \eqref{ut-proof-1} and \eqref{ut-proof-5} we see that there is a positive constant $C^*$, such that
  \begin{align}\label{ut-proof-5-2}
\begin{aligned}
\abs{\norm{\mathbf{u}_t(t)}^2_{L^2}-
\norm{\mathbf{u}_t(s)}^2_{L^2}}\leq C^*(t-s),
\end{aligned}
\end{align}
which shows that $\norm{\mathbf{u}_t(t)}^2_{L^2}$ is a non-negative, absolutely continuous function in $L^1(0,\infty)$.
Hence, by \autoref{lemma-grad-2} we have $\norm{\mathbf{u}_t(t)}^2_{L^2}\to 0\quad\text{as}\quad t\to \infty$. This implies
\eqref{theorem-2-conc-2}.

\subsubsection{Large time behavior
of $\norm{\nabla p+\rho\nabla f}_{H^{-1}}$}

 For any $\epsilon>0$ and any $ \mathbf{v}\in L^2(\Omega)$, there exists
 $ \tilde{\mathbf{v}}\in H_0^1(\Omega)$ such that $
 \norm{\mathbf{v}-\tilde{\mathbf{v}}}_{L^2}<\epsilon$.
 This infers that for any $\epsilon>0$ and any $ \mathbf{v}\in L^2(\Omega)$, we have
 \[
 \abs{\int_{\Omega}\Delta \mathbf{u}\cdot  \mathbf{v}\,d \mathbf{x}}
 \leq  \norm{\mathbf{v}-\tilde{\mathbf{v}}}_{L^2}
 \norm{\Delta \mathbf{u}}_{L^2}+
 \norm{\tilde{\mathbf{v}}}_{H^1}\norm{\mathbf{u}}_{H^1}<C_0\epsilon
 + \norm{\tilde{\mathbf{v}}}_{H^1}\norm{\mathbf{u}}_{H^1}
 <2C_0\epsilon
 \]
provided $t$ is sufficiently large. This deduces that $
\Delta \mathbf{u}\rightharpoonup 0\quad \text{in}\quad L^2(\Omega)$,
which yields \eqref{theorem-2-conc-3}. Finally, using \eqref{theorem-2-conc-1}-\eqref{theorem-2-conc-3},
for any $ \mathbf{v}\in L^2(\Omega)$
 we have
\begin{align}\label{424-proof-5-2}
\begin{aligned}
\int_{\Omega}\abs{ \mathbf{v}\cdot(\nabla P+\varrho \nabla f)}\,d\mathbf{x} 
\leq &\norm{\sqrt{\varrho +\rho_s}\mathbf{u}_t}_{L^2}
 \norm{\mathbf{v}}_{L^2}
 +C\norm{\nabla \mathbf{u}}_{L^4}\norm{\mathbf{u}}_{L^4}\norm{\mathbf{v}}_{L^2}
 \\&+ \abs{\int_{\Omega}\Delta \mathbf{u}\cdot  \mathbf{v}\,d \mathbf{x}}\to 0\quad \text{as}\quad t\to \infty,
\end{aligned}
\end{align}
where we have used
\[
 \nabla P+\frac {\rho }{\rho^*}\nabla f
= \nu \Delta \mathbf{u}
 -\frac{\partial \mathbf{u}}{\partial t}-\mathbf{u} \cdot \nabla )\mathbf{u}.
 \]
Then, the conclusion \eqref{theorem-2-conc-4} follows from \eqref{424-proof-5-2}.

\subsubsection{Large time behavior
of $\int_{\Omega}\rho f
 \,dx\,dy $}

As $t\to +\infty$, we get from \eqref{one-proof-1-n} that
\begin{align}\label{one-using}
\begin{aligned}
 2\int_{\Omega}\rho f
 \,dx\,dy \to &-2\rho^*\nu\int_0^{+\infty}\left(\int_{\Omega}\abs{\nabla \mathbf{u}(\tau)}^2\,dx\,dy\right)\,d\tau
\\&-\frac{2\rho^*\nu}{ b}
\int_0^{+\infty}\left(
\int_{\partial B(0,b)}\left(u_{\mathbf{\tau}}\right)^2 \,ds
\right)\,d\tau
\\&-2\rho^*\int_0^{+\infty}\left(\int_{\partial B(0,a)}
\left(\nu a^{-1}+\alpha\right)\left(u_{\mathbf{\tau}}\right)^2 \,ds
\right)\,d\tau\\&+
\int_{\Omega}\rho^*\abs{\mathbf{u}_0}^2\,dx\,dy
+ 2\int_{\Omega}\rho_0 f
 \,dx\,dy.
\end{aligned}
\end{align}
This yields \eqref{theorem-22-conc-4-1} and $\eqref{I-11a}_1$.

\subsubsection{Large time behavior
of $\norm{\rho +\rho_s+\gamma f-\beta}_{L^2}^2$}
 
To show \eqref{theorem-22-conc-4} and $\eqref{I-11a}_2-\eqref{I-11a}_3$.
Let us introduce the following new variables
\begin{align*}
\mathbf{v}=\mathbf{u},\quad
\rho=\theta -\rho_s+e(x,y),\quad
P=q-P_s+h(x,y),
\end{align*}
where $e(x,y)$ and $h(x,y)$ are defined by the following equations
\begin{align*}
\begin{aligned}
&e(x,y)=-\gamma f(x,y)+\beta,\quad \gamma>0,\\
&\nabla h=-e(x,y)\nabla {f},\quad \nabla P_s=-\rho_s\nabla f.
\end{aligned}
\end{align*}
Then, it can be observed that $(\mathbf{v},\theta,q)$ satisfies the following system:
   \begin{align}\label{213-1}
\begin{cases}
\frac{\partial \mathbf{v}}{\partial t}+ (\mathbf{v} \cdot \nabla )\mathbf{v} = \nu \Delta \mathbf{v}-
\frac{1}{\rho^*} \nabla q-\frac{\theta}{\rho^*} \nabla f,
 \\ \frac{\partial \theta}{\partial t}+(\mathbf{v}\cdot  \nabla)\theta
 =\gamma (\mathbf{v}\cdot  \nabla) {f},
\\ \nabla \cdot  \mathbf{v}=0,
\end{cases}
\end{align}
with the folloiwng associated boundary conditions
  \begin{align}\label{n-cond-1}
  &\mathbf{v} \cdot   \mathbf{n}|_{x^2+y^2=a^2}=0,\quad
  \mathbf{v} \cdot   \mathbf{n}|_{x^2+y^2=b^2}=\nabla \times   \mathbf{v}|_{x^2+y^2=b^2}=0,\\ \label{n-cond-2}
   &\left[\left(-\frac{q}{\rho^*}\mathbf{I} +\nu\left (\nabla \mathbf{v}+\left(\nabla \mathbf{v}\right)^{Tr} \right)\right)\cdot\mathbf{n}\right]
   \cdot   \mathbf{\tau}|_{x^2+y^2=a^2}
  +\alpha  \mathbf{v}  \cdot   \mathbf{\tau}|_{x^2+y^2=a^2}=0.
  \end{align}
  
By multiplying the first equations in  \eqref{213-1} by $\gamma \mathbf{v}$ and
the second equation in \eqref{213-1} by  $\frac{\theta}{\rho^*}$ in the $L^2$, and then summing the results, we deduce, with the aid of
  \autoref{lemma-grad-1-1}, that
\begin{align}\label{two-proof--1}
\begin{aligned}
&\frac{d}{dt}\int_{\Omega}\left(\gamma \abs{\mathbf{u}}^2
+\frac{(\rho +\rho_s+\gamma f-\beta)^2}{\rho^*}\right)
\,dx\,dy\\&=-2\gamma \nu
 \int_{\Omega}\abs{\nabla\mathbf{u}}^2\,dx\,dy
    - \red{  \frac{2\gamma \nu}{b}\int_{\partial B(0,b)}\left(u_{\mathbf{\tau}}\right)^2 \,ds
} -2\gamma \nu\int_{\partial B(0,a)}\left(a^{-1}+\frac{\alpha}{\nu}\right)\left(u_{\mathbf{\tau}}\right)^2 \,ds.
\end{aligned}
\end{align}
Integrating \eqref{two-proof--1} with respect to time yields
\begin{align}\label{two-proof--2}
\begin{aligned}
&\gamma\rho^* \norm{\mathbf{u}}_{L^2}^2
+
\norm{\rho +\rho_s+\gamma f-\beta}_{L^2}^2
+2\gamma \nu\rho^*\int_0^t\Bigg(
\norm{\nabla\mathbf{v}(t')}_{L^2}^2
+\red{  \frac{1}{b}\int_{\partial B(0,b)}\left(u_{\mathbf{\tau}}\right)^2 \,ds
} 
   \\&+\int_{\partial B(0,a)}\left(a^{-1}
   +\frac{\alpha}{\nu}\right)\left(v_{\mathbf{\tau}}(t')\right)^2 \,ds
   \Bigg)\,dt'=\gamma \rho^*\norm{\mathbf{u}_0}_{L^2}^2
+\norm{\rho_0+\gamma f-\beta}_{L^2}^2.
\end{aligned}
\end{align}
As $t\to +\infty$, we get from \eqref{two-proof--2} that
\begin{align}\label{two-proof--3}
\begin{aligned}
\norm{\rho +\rho_s+\gamma f-\beta}_{L^2}^2\to 
  &-2\gamma \nu\rho^*\int_0^{+\infty}\Bigg(
\norm{\nabla\mathbf{v}(t')}_{L^2}^2
+\red{  \frac{1}{b}\int_{\partial B(0,b)}\left(u_{\mathbf{\tau}}\right)^2 \,ds
} 
   \\&+\int_{\partial B(0,a)}\left(a^{-1}
   +\frac{\alpha}{\nu}\right)\left(v_{\mathbf{\tau}}(t')\right)^2 \,ds
   \Bigg)\,dt'\\
   &+\gamma \rho^*\norm{\mathbf{u}_0}_{L^2}^2
+\norm{\rho_0+\gamma f-\beta}_{L^2}^2.
\end{aligned}
\end{align}
This shows \eqref{theorem-22-conc-4} and $\eqref{I-11a}_2$. Finally,
$\eqref{I-11a}_3$ follows
from \eqref{one-using} and \eqref{two-proof--3}.

\subsection{Proof of \autoref{theorem-3-3}}

We get from \eqref{one-using} that $\int_{\Omega}\rho f
 \,dx\,dy \to 0$ if and only if for any $\gamma>0$, one has
\begin{align*}
\begin{aligned}
&\gamma\int_{\Omega}\rho^*\abs{\mathbf{u}_0}^2\,dx\,dy
+ 2\gamma\int_{\Omega}\rho_0 f
 \,dx\,dy\\&=2\gamma\rho^*\nu\int_0^{+\infty}\left(\int_{\Omega}\abs{\nabla \mathbf{u}(\tau)}^2\,dx\,dy\right)\,d\tau
+\frac{2\gamma\rho^*\nu}{ b}
\int_0^{+\infty}\left(
\int_{\partial B(0,b)}\left(u_{\mathbf{\tau}}\right)^2 \,ds
\right)\,d\tau
\\&\quad+2\gamma\rho^*\int_0^{+\infty}\left(\int_{\partial B(0,a)}
\left(\nu a^{-1}+\alpha\right)\left(u_{\mathbf{\tau}}\right)^2 \,ds
\right)\,d\tau
\\&=\gamma \rho^*\norm{\mathbf{u}_0}_{L^2}^2
+\norm{\rho_0+\gamma f-\beta}_{L^2}^2-\lim_{t\to+\infty}
\norm{\rho +\rho_s+\gamma f-\beta}_{L^2}^2
\end{aligned}
\end{align*}
which is equivalent to 
\begin{align}\label{one-usingsb}
\begin{aligned}
& 2\gamma\int_{\Omega}\rho_0 f
 \,dx\,dy=\norm{\rho_0+\gamma f-\beta}_{L^2}^2-\lim_{t\to+\infty}
\norm{\rho +\rho_s+\gamma f-\beta}_{L^2}^2.
\end{aligned}
\end{align}
This proves the first conclusion of \autoref{theorem-3-3}.

We get from \eqref{two-proof--3} that 
\[
\norm{\rho +\rho_s+\gamma f-\beta}_{L^2}^2 \to 0,
\]
 if and only if there exists $\gamma>0$ and $\beta$ such that
\begin{align}\label{two-proof--3}
\begin{aligned}
& \gamma \rho^*\norm{\mathbf{u}_0}_{L^2}^2
+\norm{\rho_0+\gamma f-\beta}_{L^2}^2
 \\& =2\gamma \nu\rho^*\int_0^{+\infty}\Bigg(
\norm{\nabla\mathbf{v}(t')}_{L^2}^2
+\red{  \frac{1}{b}\int_{\partial B(0,b)}\left(u_{\mathbf{\tau}}\right)^2 \,ds
} 
   \\&\quad+\int_{\partial B(0,a)}\left(a^{-1}
   +\frac{\alpha}{\nu}\right)\left(v_{\mathbf{\tau}}(t')\right)^2 \,ds
   \Bigg)\,dt'\\
   &=\gamma \int_{\Omega}\rho^*\abs{\mathbf{u}_0}^2\,dx\,dy
+ 2\gamma \int_{\Omega}\rho_0 f
 \,dx\,dy-\gamma \lim_{t\to+\infty}2\int_{\Omega}\rho f
 \,dx\,dy.
\end{aligned}
\end{align}
which is equivalent to 
\begin{align}\label{one-usingsb-2}
\begin{aligned}
\norm{\rho_0+\gamma f-\beta}_{L^2}^2= 2\gamma \int_{\Omega}\rho_0 f
 \,dx\,dy-\gamma \lim_{t\to+\infty}2\int_{\Omega}\rho f
 \,dx\,dy.
\end{aligned}
\end{align}
This proves the second conclusion of \autoref{theorem-3-3}.

\subsection{Proof of \autoref{theorem-4}}

Under the conditions of \autoref{theorem-4}, we employ similar arguments to the proof of  \autoref{ll2}-\autoref{21336}
for \eqref{213} to get similar lemmas for the equations \eqref{213-2}. Thus, we make use of Gagliardo-Nirenberg's interpolation
inequality $\norm{\mathbf{u}}_{W^{1,p}}^p\leq \norm{\mathbf{u}}_{H^1}^{2}\norm{\mathbf{u}}_{H^2}^{p-2}$ ($p\geq 2$)
to find that the solutions of equations \eqref{213-2} satisfy
\begin{subequations}
     \begin{align}\label{theorem-6-conc-1}
&\mathbf{u} \in L^{\infty}\left(\left(0,\infty\right);H^2(\Omega)\right)
\cap L^{p}\left(\left(0,\infty\right);W^{1,p}(\Omega)\right),\quad 2\leq p<\infty,\\
\label{theorem-6-conc-2}
&\mathbf{u}_t \in
L^{\infty}\left(\left(0,\infty\right);L^2(\Omega)\right)
\cap L^{2}\left(\left(0,\infty\right);H^{1}(\Omega)\right),\\
\label{theorem-6-conc-3}
&p\in L^{\infty}\left(\left(0,\infty\right);H^1(\Omega)\right),\\
\label{theorem-6-conc-4}
&\rho \in L^{\infty}\left(\left(0,\infty\right);L^2(\Omega)\right),
 \end{align}
   \end{subequations}
and
\begin{subequations}
     \begin{align}\label{theorem-7-conc-1}
&\norm{\mathbf{u} }_{W^{1,s}}\to 0,\quad t\to \infty, \quad 1\leq s<\infty,\\
\label{theorem-7-conc-2}
&\norm{\mathbf{u}_t }_{L^{2}}\to 0,\quad t\to \infty, \\
\label{theorem-7-conc-5}
&\norm{\rho}_{L^{2}} \to C_0,\quad t\to \infty.
 \end{align}
 \end{subequations}
 Hence, \eqref{theorem-4-conc-1} follows from \eqref{theorem-6-conc-1}, and
 \eqref{theorem-4-conc-3} is obtained from \eqref{theorem-6-conc-3}.

 \subsubsection{Estimate of $\norm{\nabla \rho}_{ L^{+\infty}\left((0,\infty);L^{2}(\Omega)\right)}$ for the linear problem}

 In polar coordinates, we see that
$\rho$ and $\nabla \rho =(\partial_r\rho,\partial_{\theta}\rho/r)$
and $\omega$ solve 
\begin{align}\label{example}
\begin{cases}
  \frac{\partial \rho}{\partial t}= -h_0 gu_r,\\
   \frac{\partial \nabla \rho}{\partial t}= -h_0 g\nabla u_r,\\
\frac{\partial \omega}{\partial t}  = \nu \Delta \omega+\frac{g}{r}\partial_{\theta}\rho,\\
\omega=0,\quad r=a, b,\\
\omega|_{t=0}= \omega_0,\\
\rho|_{t=0}= \rho_0,\quad \rho_0|_{r=a,b}=0.
\end{cases}
\end{align}

Multiplying $\eqref{example}_1$
and $\eqref{example}_2$
 by $r^2\frac{\nabla \rho}{-h_0}$ and $r^2\omega$ in $L^2$, integrating by parts, we have
 \begin{align}\label{twn-find}
  \begin{aligned}
  &\frac{d}{dt}
 \int_a^b\left(
 \int_0^{2\pi}
  \left(
 \frac{1}{-2h_0}\left(\partial_r\rho \right)^2+
 \frac{1}{-2h_0r^2} \left(\partial_{\theta}\rho\right)^2
 + \frac{\omega^2}{2}
 \right)
 r^3\,d\theta
 \right)\,dr
 \\&+\nu\int_a^b\left(
 \int_0^{2\pi}
  \left(
\omega \Delta \omega 
 \right)
 r^3\,d\theta
 \right)\,dr\\
& =g \int_a^b\left( \int_0^{2\pi}
  \left(
\partial_{r}u_{r}\partial_r\rho+
 \frac{1}{r^2} \partial_{\theta}u_{r}\partial_{\theta}\rho
 \right)
 r^3\,d\theta
 \right)\,dr\\
 &\quad+g \int_a^b\left( \int_0^{2\pi}
\partial_{\theta}\rho
  \left(
\partial_ru_{\theta}+\frac{u_{\theta}}{r}
-\frac{\partial_{\theta}u_r}{r}
 \right)
 \,d\theta
 \right)r^2\,dr\\
 &=g \int_a^b\left( \int_0^{2\pi}
  \left(
\partial_{r}u_{r}\partial_r\rho+
\frac{\partial_{\theta}\rho}{r}
  \left(
\partial_ru_{\theta}+\frac{u_{\theta}}{r}\right)
 \right)
 r^3\,d\theta
 \right)\,dr
 \\
 &=g \int_a^b\left( \int_0^{2\pi}
  \left(
-\partial_r\rho
\left(
u_r
+\partial_{\theta}u_{\theta}
\right)
+\partial_{\theta}\rho
  \left(
\partial_ru_{\theta}+\frac{u_{\theta}}{r}\right)
 \right)r^2
 \,d\theta
 \right)\,dr\\
 &=g \int_a^b\left( \int_0^{2\pi}
  \left(
-\partial_r\rho
u_r
+\partial_{\theta}\rho
\frac{u_{\theta}}{r}
+\frac{2\rho}{r}\partial_{\theta}u_{\theta}
\right)r^2
 \,d\theta
 \right)\,dr
 \\
 &=g \int_a^b\left( \int_0^{2\pi}
   \left(
-\partial_r\rho
u_r
-\rho
\frac{\partial_{\theta}u_{\theta}}{r}
+\frac{2\rho}{r}\partial_{\theta}u_{\theta}
\right)r^2
 \,d\theta
 \right)\,dr\\
  &=g \int_a^b\left( \int_0^{2\pi}
   \left(
-\partial_r\rho
u_r
+\frac{\rho}{r}\partial_{\theta}u_{\theta}
\right)r^2
 \,d\theta
 \right)\,dr\\
  &=g \int_a^b\left( \int_0^{2\pi}
r\rho u_r
 \,d\theta
 \right)\,dr.
 \end{aligned}
  \end{align}
 Note that
 \[
  \begin{aligned}
& \begin{aligned}
 & \frac{d}{dt}
  \int_a^b\int_0^{2\pi}
  \frac{r\rho^2}{-2h_0}
   \,d\theta\,dr
  =g
    \int_a^b\int_0^{2\pi}r
  \rho u_r \,d\theta\,dr,
  \end{aligned}\\
 & \begin{aligned}
 \nu\int_a^b\left(
 \int_0^{2\pi}
  \left(
\omega \Delta \omega 
 \right)
 r^3\,d\theta
 \right)\,dr
 =-&\nu 
 \int_a^b\left(
 \int_0^{2\pi}
  \left(
\left(\partial_r\omega\right)^2+
\frac{1}{r^2}
\left(\partial_{\theta}\omega \right)^2
 \right)
 r^3\,d\theta
 \right)\,dr
 \\&-2\nu 
 \int_a^b\left(
 \int_0^{2\pi}
\omega \partial_r\omega
 r^2\,d\theta
 \right)\,dr.
   \end{aligned}
     \end{aligned}
 \]
 This, together with \eqref{twn-find}, gives
 \begin{align}\label{newi-find-two}
  \begin{aligned}
  &\frac{1}{2}\frac{d}{dt}
 \int_a^b\left(
 \int_0^{2\pi}
  \left(
 \left(\partial_r\rho \right)^2+
 \frac{1}{r^2} \left(\partial_{\theta}\rho\right)^2
 +\omega^2
 \right)
 r^3\,d\theta
 \right)\,dr
 +\frac{d}{dt}
  \int_a^b\int_0^{2\pi}
  \frac{r\rho^2}{-2h_0}
   \,d\theta\,dr
 \\&+\frac{\nu}{2}
 \int_a^b\left(
 \int_0^{2\pi}
  \left(
\left(\partial_r\omega\right)^2+
\frac{1}{r^2}
\left(\partial_{\theta}\omega \right)^2
 \right)
 r^3\,d\theta
 \right)\,dr
 \leq C
 \int_a^b\left(
 \int_0^{2\pi}
\omega^2
 r\,d\theta
 \right)\,dr.
 \end{aligned}
  \end{align}
  Integrating \eqref{newi-find-two} with respect time, we can get that
 \[
 \norm{\nabla \rho}_{L^2}^2
 \in L^{\infty}((0,\infty);L^2(\Omega)).
 \]

 \subsubsection{Estimate of $\norm{\mathbf{u}_t}_{ L^{2}\left((0,\infty);H^{2}(\Omega)\right)}$ for the linear problem}

 We now derive a bound for $\norm{\mathbf{u}_t}_{L^{2}\left((0,\infty);H^{2}(\Omega)\right)}$.
 For the linear problem \eqref{213-2}, $\omega_t$ solves
    \begin{align}\label{three-proof--4-w-1}
\begin{cases}
\omega_{tt}= \nu \Delta \omega_t-\frac{1}{\rho^*}\nabla
\left( h(x,y)(\mathbf{u}\cdot  \nabla  {gr})\right)
 \cdot \nabla^{\perp}gr ,\\
\omega_t=0,\quad x^2+y^2=a^2,b^2,\\
\omega|_{t=0}= \omega_0.
\end{cases}
\end{align}
 Multiplying the first equation of \eqref{three-proof--4-w-1} by $\omega_t$ in $L^2$, one has
    \begin{align}\label{three-proof--4-w-2}
\begin{aligned}
\frac{1}{2}\frac{d}{dt}\norm{\omega_t}_{L^2}^2=&-
\nu \norm{\nabla \omega_t}
_{L^2}^2-
\frac{1}{\rho^*}
\int_{\Omega}\omega_t \nabla \left( h(x,y)(\mathbf{u}\cdot  \nabla  {gr})\right)
 \cdot \nabla^{\perp}gr \,dx\,dy
 \\&\leq 
 -\nu \norm{\nabla \omega_t}
_{L^2}^2+
 C\norm{\mathbf{u}_t}_{H^1}
 \norm{\mathbf{u}}_{H^1}<+\infty.
\end{aligned}
\end{align}
Integrating the above inequality with respect to time and noting that the right-hand side is uniformly bounded in time due to \eqref{theorem-6-conc-1}–\eqref{theorem-6-conc-2}, we find that
\begin{align}\label{three-proof--4-w-8-1}
\mathbf{u}_t\in L^{2}\left((0,\infty);H^{2}(\Omega)\right).
\end{align}
From
\[
 \norm{\mathbf{u}_t}_{H^1}^2 \leq C
 \norm{\omega_t}_{L^2}^2,
\]
it follows that
\begin{align}\label{three-proof--4-w-8}
\mathbf{u}_t\in L^{\infty}\left((0,\infty);H^{1}(\Omega)\right).
\end{align}
Hence, the properties \eqref{three-proof--4-w-8-1} and \eqref{three-proof--4-w-8} gives \eqref{theorem-4-conc-2}.



 \subsubsection{Large time behavior of $\norm{\mathbf{u}}_{H^{2}}$}
 For the linear problem \eqref{213-2}, $\omega$ solves
   \begin{align}\label{three-proof--3-1}
\begin{cases}
\frac{\partial \omega}{\partial t}  = \nu \Delta \omega+\frac{1}{\rho^*}\nabla\rho \cdot \nabla^{\perp}gr ,\\
\omega=0,\quad x^2+y^2=a^2,b^2,\\
\omega|_{t=0}= \omega_0.
\end{cases}
\end{align}

Multiplying \eqref{three-proof--3-1} by $\omega$ in $L^2$ and integrating by parts, we have
\begin{align}\label{proof-43-2-2-1}
\begin{aligned}
\frac{1}{2}\frac{d}{dt}\norm{\omega}^2_{L^2}=
& -\nu \norm{\nabla \omega}^2_{L^2}
+\frac{1}{\rho^*}\int_{\Omega}\rho
\nabla\omega \cdot \nabla^{\perp}gr\,dx\,dy.
\end{aligned}
\end{align}

Based on \eqref{proof-43-2-2-1} and \autoref{lemma-grad-1}, we have
\begin{align}\label{proof-43-2-2}
\begin{aligned}
\nu \norm{\mathbf{u}}^2_{H^2}&\leq \nu C
\norm{\nabla \omega}^2_{L^2}
+\nu C \norm{\mathbf{u}}^2_{H^1}
\\&=-C
\int_{\Omega}
\omega \omega_t
\,dx\,dy
+\frac{C}{\rho^*}\int_{\Omega}\rho
\nabla\omega \cdot \nabla^{\perp}gr
\,dx\,dy+\nu C \norm{\mathbf{u}}^2_{H^1}\\
\leq &C\left(
 \norm{\mathbf{u}_t}_{H^1}
 + \norm{\rho}_{H^1}
\right) \norm{\mathbf{u}}_{H^1}
+C\nu \norm{\mathbf{u}}_{H^1}^2.
\end{aligned}
\end{align}
Using \eqref{theorem-6-conc-3}, \eqref{three-proof--4-w-8} and the uniform bound on $\norm{\nabla \rho}_{L^2}$, we observe that the terms inside the parentheses on the right‑hand side of \eqref{proof-43-2-2} are uniformly bounded in time. Consequently, \eqref{theorem-3-conc-1} follows from \eqref{theorem-7-conc-1}.
 \subsubsection{Large time behavior
of $\norm{\mathbf{u}_t}_{H^{1}}$}

From \eqref{three-proof--4-w-2}, we get that
\begin{align}\label{utl2h21}
\norm{\omega_t}_{L^2}\to 0~\text{as}~t\to \infty.
\end{align}
Recalling that
\[
\norm{\mathbf{u}_t}^2_{H^1}\leq
 C\norm{\omega_t}^2_{L^2}
+ C \norm{\mathbf{u}_t}^2_{L^2}
\]
we use \eqref{utl2h21} and \eqref{theorem-7-conc-2} to deduce
\[
\norm{\mathbf{u}_t}^2_{H^1} \to 0~\text{as}~t\to \infty ,
\]
which yields \eqref{theorem-3-conc-2}. Finally, the limits \eqref{theorem-3-conc-1}-\eqref{theorem-3-conc-2} 
and the equations \eqref{213-2} yield \eqref{theorem-3-conc-3}.

\section{Appendix}

\begin{appendix}

   \begin{lemma}\label{lemma-grad}
   For $\mathbf{u}\in W^{1,p}(\Omega)$ satisfying $\mathbf{u}\cdot \mathbf{n}|_{\partial \Omega}=0$,  there exists constants $ C_a^b$ and $ D_a^b$ such that
             \begin{align}\label{omega}
           \begin{aligned}
 C_a^b\norm{\mathbf{u}}_{L^p}^p\leq \int_{\Omega}\abs{\nabla \mathbf{u}}^p\,dxdy+\int_{\partial B(0,b)}\abs{u_{\tau}}^p\,ds.
   \end{aligned}
     \end{align}
     and
                 \begin{align}\label{omega}
           \begin{aligned}
 D_a^b\norm{u}_{W^{1,p}}^2\leq
 \int_{\Omega}\abs{\nabla u}^p\,dxdy
 +\int_{\partial B(0,b)}\abs{u_{\tau}}^p\,ds.
   \end{aligned}
     \end{align}
  \end{lemma}

\begin{figure}[tbh]
\centering
        \begin{tikzpicture}[>=stealth',xscale=1,yscale=1,every node/.style={scale=0.8}]
\draw [thick,->] (0,-2.5) -- (0,2.5) ;
\draw [thick,->] (-2.5,0) -- (2.5,0) ;
\node [below right] at (3.1,0) {$x$};
\draw [thick,-,blue] (-1.732,1) -- (1.732,1) ;
\draw [thick,-,blue] (-1.732,-1) -- (1.732,-1) ;
\node [below right] at (1,0) {$a$};
\node [below right] at (2,0) {$b$};
\node [above right] at (0,1.3) {$\Omega_1$};
\node [above left] at (-1.3,0) {$\Omega_2$};
\node [above right] at (1.3,0) {$\Omega_3$};
\node [below right] at (0,-1.3) {$\Omega_4$};
\node [left] at (0,3.1) {$y$};
\draw[domain = -2:360,blue][samples = 200] plot({cos(\x)}, {sin(\x)});
\draw[domain = -2:360,blue][samples = 200] plot({2*cos(\x)}, {2*sin(\x)});
\end{tikzpicture}
\caption{$\mathbf{g}=-g\left(\frac{x}{r}, \frac{y}{r}\right)^T$}
\label{domain14}
\end{figure}
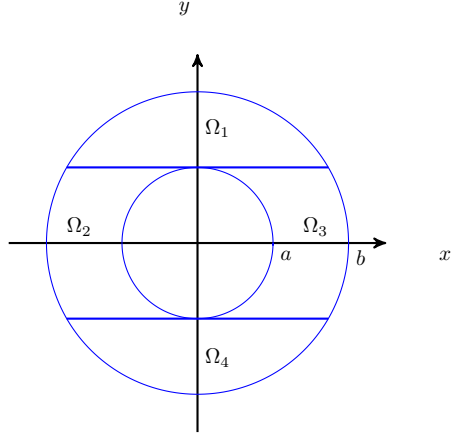

  \begin{proof}
    Let $ \Omega=\cup_{j=1}^4\Omega_j$, shown in \autoref{domain14}.
   For $(x,y)\in \Omega_1\cup  \Omega_4$, we define $l^{j1}_{x,y}$ and $l^{j2}_{x,y}$ to be the two horizontal lines as follows.
   \begin{align}
    \begin{aligned}
   & l^{11}_{x,y}=l^{41}_{x,y}:=\{(s,y)|-\sqrt{b^2-y^2}\leq s\leq x\},\\
   &l^{12}_{x,y}=l^{42}_{x,y}:=\{(s,y)| -\sqrt{b^2-y^2}\leq x\leq \sqrt{b^2-y^2}\},
   \end{aligned}
      \end{align}
      For $(x,y)\in \Omega_2$, we define $l^1_{x,y}$ and $l^2_{x,y}$
    be the two
   horizontal lines as follows
   \begin{align}
   \begin{aligned}
   &l^{21}_{x,y}:=\{(s,y)|-\sqrt{b^2-y^2}\leq s\leq x\},\\
    &l^{22}_{x,y}:=\{(s,y)| -\sqrt{b^2-y^2}\leq x\leq -\sqrt{a^2-y^2}\}.
       \end{aligned}
   \end{align}
     For $(x,y)\in \Omega_3$, we define $l^1_{x,y}$ and $l^2_{x,y}$ to be the two horizontal lines as follows.
  \begin{align}
   \begin{aligned}
   &l^{31}_{x,y}:=\{(s,y)|x\leq s\leq \sqrt{b^2-y^2}\},\\
   & l^{32}_{x,y}:=\{(s,y)| \sqrt{a^2-y^2}\leq s\leq \sqrt{b^2-y^2}\}.
      \end{aligned}
   \end{align}
We take $j=1$ and $j=2$ as an example. The fundamental theorem of calculus yields 
   \begin{align}
 \mathbf{u}(x,y)= \mathbf{u}(-\sqrt{b^2-y^2},y)+
 \int_{l^{11}_{x,y}}\partial_x \mathbf{u}\,ds.
   \end{align}
   and 
\begin{align}
 \mathbf{u}(x,y)= \mathbf{u}(-\sqrt{b^2-y^2},y)+
 \int_{l^{21}_{x,y}}\partial_x \mathbf{u}\,ds.
   \end{align}
  which gives
        \begin{align}\label{omega-14}
           \begin{aligned}
&\int_{\Omega_1}\abs{\mathbf{u}(x,y)}^p\,dxdy
=\int_{a}^{b}
\left(\int_{-\sqrt{b^2-y^2}}^{\sqrt{b^2-y^2}}
\abs{\mathbf{u}(x,y)}^p\,dx\right)\,dy
\\&\leq 2^p\sqrt{b^2-a^2}\int_{\partial B(0,b)}\abs{u_{\tau}}^p\,ds
+ \int_{\Omega_1}\left(
 (2\sqrt{b^2-y^2})^{p-1}
 \int_{l^{12}_{x,y}}\abs{\partial_x \mathbf{u}}^p\,ds\right)\,dxdy
 \\&\leq   2^p\sqrt{b^2-a^2}\int_{\partial B(0,b)}\abs{u_{\tau}}^p\,ds+
 \int_{\Omega_1}\left((2\sqrt{b^2-y^2})^{p-1}
 \int_{l^{12}_{x,y}}\abs{\nabla \mathbf{u}}^p\,ds\right)\,dxdy
\\& \leq  2^p\sqrt{b^2-a^2}\int_{\partial B(0,b)}\abs{u_{\tau}}^p\,ds+2b(2\sqrt{b^2-a^2})^{p-1}
 \int_{\Omega_1}\abs{\nabla \mathbf{u}}^p\,dxdy.
   \end{aligned}
     \end{align}
 and 
         \begin{align}\label{omega-23}
           \begin{aligned}
&\int_{\Omega_2}\abs{\mathbf{u}(x,y)}^p\,dxdy
=\int_{-a}^{a}
\left(\int_{-\sqrt{b^2-y^2}}^{-\sqrt{a^2-y^2}}
\abs{\mathbf{u}(x,y)}^p\,dx\right)\,dy
\\&\leq 2^p\sqrt{b^2-a^2}\int_{\partial B(0,b)}\abs{u_{\tau}}^p\,ds
+ \int_{\Omega_2}\left(
 (\sqrt{b^2-a^2})^{p-1}
 \int_{l^{21}_{x,y}}\abs{\partial_x \mathbf{u}}^p\,ds\right)\,dxdy
 \\&\leq   2^p\sqrt{b^2-a^2}\int_{\partial B(0,b)}\abs{u_{\tau}}^p\,ds+
 \int_{\Omega_2}\left((\sqrt{b^2-a^2})^{p-1}
 \int_{l^{21}_{x,y}}\abs{\nabla \mathbf{u}}^p\,ds\right)\,dxdy
\\& \leq  2^p\sqrt{b^2-a^2}\int_{\partial B(0,b)}\abs{u_{\tau}}^p\,ds+2b(\sqrt{b^2-a^2})^{p-1}
 \int_{\Omega_2}\abs{\nabla \mathbf{u}}^p\,dxdy.
   \end{aligned}
     \end{align}

  Hence, in view of \eqref{omega-14} and \eqref{omega-23}, there exists a positive constant $C_a^b$
  such that
           \begin{align}\label{omega}
           \begin{aligned}
\int_{\Omega}\abs{\mathbf{u}(x,y)}^p\,dxdy\leq
C_a^b\left(
\int_{\partial B(0,b)}\abs{u_{\tau}}^p\,ds+
 \int_{\Omega}\abs{\nabla \mathbf{u}}^p\,dxdy
 \right).
   \end{aligned}
     \end{align}
     \end{proof}   

\begin{lemma}\label{lemma-grad-1-1}   
      For $\mathbf{v}, \mathbf{w}\in H^{2}(\Omega)$ satisfying boundary conditions \eqref{cond-1}-\eqref{cond-2} and $\nabla \cdot \mathbf{v}=0$,  we have
       \begin{align*}
    \begin{aligned}
 -\int_{\Omega}\Delta \mathbf{v}\cdot\mathbf{w}\,dx\,dy
   =& \int_{\Omega}\nabla\mathbf{v}:\nabla\mathbf{w}\,dx\,dy
   +\int_{\partial B(0,a)}\left(a^{-1}+\frac{\alpha}{\nu}\right)v_{\mathbf{\tau}} w_{\mathbf{\tau}}\,ds+\frac{1}{b}\int_{\partial B(0,b)}v_{\mathbf{\tau}} w_{\mathbf{\tau}} \,ds.
   \end{aligned}
  \end{align*}
\end{lemma}
 \begin{proof}
 Let us denote $(\partial_1,\partial_2):=(\partial_x,\partial_y)$.
 Taking integration by parts and projecting $\mathbf{v}$ onto its
tangential and normal part on the boundary, i.e.,
\begin{align*}
&\mathbf{v}=v_{\mathbf{n}}\mathbf{n}+
v_{\tau}\mathbf{\tau},\quad 
v_{\mathbf{n}}=\mathbf{v}\cdot\mathbf{n},\quad
v_{\tau}=\mathbf{v}\cdot\mathbf{\tau},
\\ &\mathbf{w}=w_{\mathbf{n}}\mathbf{n}+
w_{\tau}\mathbf{\tau},\quad 
w_{\mathbf{n}}=\mathbf{w}\cdot\mathbf{n},\quad
w_{\tau}=\mathbf{w}\cdot\mathbf{\tau},
\end{align*}
we have
 \begin{align}
  \begin{aligned}
  -\int_{\Omega}\Delta \mathbf{v}\cdot\mathbf{w}\,dx\,dy&=
  -\int_{\partial\Omega}w_j n_i\partial_iv_j\,dx\,dy
  +\int_{\Omega}\nabla\mathbf{v}:\nabla\mathbf{w}\,dx\,dy\\
 & =-\int_{\partial\Omega}w_k\tau_k\tau_jn_i\partial_iv_j\,ds-
  \int_{\partial\Omega}w_kn_kn_jn_i\partial_iv_j\,ds
  +\int_{\Omega}\nabla\mathbf{v}:\nabla\mathbf{w}\,dx\,dy\\
   &=-\int_{\partial\Omega}w_k\tau_k\tau_jn_i\left(\partial_iv_j
   +\partial_jv_i\right)\,ds+\int_{\partial\Omega}w_k\tau_k\tau_jn_i\partial_jv_i\,ds\\
  &\quad +\int_{\Omega}\nabla\mathbf{v}:\nabla\mathbf{w}\,dx\,dy\\
  &=\int_{\Omega}\nabla\mathbf{v}:\nabla\mathbf{w}\,dx\,dy
   +\int_{\partial B(0,a)}\left(a^{-1}+\frac{\alpha}{\nu}\right)w_{\mathbf{\tau}}v_{\mathbf{\tau}} \,ds\\
     &\quad+\frac{1}{b}\int_{\partial B(0,b)}w_{\mathbf{\tau}}v_{\mathbf{\tau}} \,ds
 \end{aligned}
  \end{align}
  where we have used $\mathbf{v}\cdot\mathbf{n}|_{\partial\Omega}=0$, $\nu\left(\nabla\mathbf{v}+\left(\nabla\mathbf{v}\right)^{\rm{T_{r}}}\right)\cdot\mathbf{n}\cdot\mathbf{\tau}|_{x^2+y^2=a^2}=-\alpha\mathbf{v}\cdot\mathbf{\tau}|_{x^2+y^2=a^2}$, $\nabla \times   \mathbf{v}|_{x^2+y^2=b^2}=0$,
  $\mathbf{w}\cdot\mathbf{n}|_{\partial\Omega}=0$, $\nabla \times   \mathbf{w}|_{x^2+y^2=b^2}=0$
   and
 \begin{align}\label{oooo}
  \begin{aligned}
 & \int_{\partial B(0,a)}w_k\tau_k\tau_jn_i\left(\partial_iv_j
   +\partial_jv_i\right)\,ds   =-
  \frac{\alpha}{\nu} \int_{\partial B(0,a)}v_{\mathbf{\tau}}w_{\mathbf{\tau}}  \,ds\\
  &  \int_{\partial B(0,b)}w_k\tau_k\tau_jn_i\left(\partial_iv_j
   +\partial_jv_i\right)\,ds   =-
  \frac{2}{b} \int_{\partial B(0,b)}v_{\mathbf{\tau}}w_{\mathbf{\tau}} \,ds,\\
  &\int_{\partial B(0,a)}w_k\tau_k\tau_jn_i\partial_jv_i\,ds
  =\frac{1}{a}
   \int_{\partial B(0,a)}v_{\mathbf{\tau}}w_{\mathbf{\tau}} \,ds,\\
    &\int_{\partial B(0,b)}w_k\tau_k\tau_jn_i\partial_jv_i\,ds
  =-\frac{1}{b}
   \int_{\partial B(0,b)}v_{\mathbf{\tau}}w_{\mathbf{\tau}} \,ds.
 \end{aligned}
  \end{align}
In the process deriving the \eqref{oooo}, we have utilized
  \[
\tau_jn_i\partial_jv_i|_{\partial B(0,a)}=
 \mathbf{n}\cdot(\mathbf{\tau}\cdot\nabla) (v_{\mathbf{\tau}}\mathbf{\tau})
|_{\partial B(0,a)}
=\left(\mathbf{n}\cdot(\mathbf{\tau}\cdot\nabla) \mathbf{\tau}
|_{\partial B(0,a)}\right)
v_{\mathbf{\tau}}=\frac{1}{a}v_{\mathbf{\tau}},
  \]
  and
      \[
\tau_jn_i\partial_jv_i|_{\partial B(0,b)}=
 \mathbf{n}\cdot(\mathbf{\tau}\cdot\nabla) (v_{\mathbf{\tau}}\mathbf{\tau})
|_{\partial B(0,b)}
=\left(\mathbf{n}\cdot(\mathbf{\tau}\cdot\nabla) \mathbf{\tau}
|_{\partial B(0,b)}\right)
v_{\mathbf{\tau}}=-\frac{1}{b}v_{\mathbf{\tau}}.
  \]
  
 \end{proof}
      \begin{lemma}\label{lemma-grad-1}
   For $\mathbf{u}\in W^{k,p}(\Omega)$ satisfying
   $\mathbf{u} \cdot   \mathbf{n}|_{x^2+y^2=a^2, b^2}=0$
   and $\nabla \cdot \mathbf{u}=0$, we have
            \begin{align}
           \begin{aligned}
\norm{\nabla \mathbf{u}}_{W^{k,p}}\leq C\norm{\omega}_{W^{k,p}}
,\quad 1<p<\infty.
   \end{aligned}
     \end{align}
  \end{lemma}
   \begin{proof}
   Letting $u_1=-\partial_y\psi$, $u_2=\partial_x\psi$, then we have
               \begin{align}
           \begin{cases}
\Delta \psi=\omega, \quad (x,y)\in \Omega,\\
 \psi=\beta_b,\quad \quad x^2+y^2=b^2,\\
  \psi=\beta_a,\quad \quad x^2+y^2=a^2,
   \end{cases}
     \end{align}
     where $\beta_a$ and $\beta_b$ are two constants. Replacing $\psi$ by
     \[
     \psi+\beta_b- \frac{(\beta_b-\beta_a)\ln b}{\ln\left(\frac{b}{a}\right)}
     +
     \frac{\beta_b-\beta_a}{\ln \left(\frac{b}{a}\right)}\ln r ,
     \]
    we get that
               \begin{align}
           \begin{cases}
\Delta \psi=\omega, \quad (x,y)\in \Omega,\\
 \psi=0,\quad \quad x^2+y^2=b^2,\\
  \psi=0,\quad \quad x^2+y^2=a^2.
   \end{cases}
     \end{align}
     By the elliptic regularity theory, we conclude
                 \begin{align}
           \begin{aligned}
\norm{\nabla \mathbf{u}}_{W^{k,p}}\leq \norm{\psi}_{W^{k+2,p}} \leq C\norm{\omega}_{W^{k,p}}.
   \end{aligned}
     \end{align}
   \end{proof}
   
 \begin{lemma}[\cite{Charles-R2018}]
 \label{lemma-grad-2}
  Assume $g(t)\in L^1(0,\infty)$ is a non-negative and uniformly continuous function. Then,
\[
g(t)\to 0\quad \text{as}\quad t\to \infty.
\]
In particular, if $g(t)\in L^1(0,\infty)$ is a non-negative and satisfies that for any $0\leq s<t<\infty$, such that
$\abs{g(t)-g(s)}\leq C(t-s)$ for some constant $C$, then
\[
g(t)\to 0\quad \text{as}\quad t\to \infty.
\]
  \end{lemma}

 \begin{lemma}[\cite{Charles-R2018}] \label{lemma-grad-3}
 Assume $X(t)$, $Y(t)$, $A(t)$ and $B(t)$ are non-negative functions, satisfying
 \[
 X'+Y\leq AX+BX\ln (1+Y).
 \]
 Let $T>0$, and $A(t)\in L^1(0,T)$ and $B(t)\in L^2(0,T)$. Then, for any $t\in [0,T]$ we have
 \[
 X(t)\leq (1+X(0)^{e^{\int_0^tB(\tau)\,d\tau}}
 e^{
 \int_0^t e^{\int_s^tB(\tau)\,d\tau}(A(s)+B^2(s))\,ds}.
 \]
 and
 \[
 \int_0^t Y(\tau)\,d\tau \leq X(t) \int_0^t A(\tau)\,d\tau
 +X^2(t) \int_0^t B^2(\tau)\,d\tau <\infty .
 \]
  \end{lemma}

\end{appendix}

 \section*{Acknowledgement}

Q. Wang was supported by the Natural Science Foundation
of Sichuan Province (No.2025ZNSFSC0072).

    \section*{Conflict Of Interest Statement}
We declare that we have no financial and personal relationships with
other people or organizations that can inappropriately influence our
work, there is no professional or other personal interest of any nature
or kind in any product, service and/or company that could be
construed as influencing the position presented in, or the review of,
the manuscript.

    \section*{Data Availability Statement}
This manuscript has no associated data.

 \itemsep=0pt

\end{document}